\documentclass{amsart}
\pdfoutput=1
\usepackage{amsmath}
\usepackage[alphabetic]{amsrefs}
\usepackage{microtype}
\usepackage{lmodern}
\usepackage{hyperref}

\theoremstyle{plain}
\newtheorem{theorem}[subsection]{Theorem}
\newtheorem{proposition}[subsection]{Proposition}
\newtheorem{lemma}[subsection]{Lemma}
\newtheorem{corollary}[subsection]{Corollary}

\theoremstyle{definition}
\newtheorem{definition}[subsection]{Definition}
\newtheorem{example}[subsection]{Example}

\newtheorem{question}[subsection]{Question}

\theoremstyle{remark}
\newtheorem{remark}[subsection]{Remark}
\newtheorem{remarks}[subsection]{Remarks}

\numberwithin{equation}{section}

\newcommand{\cal}{\mathcal}

\newcommand{\clc}{{\cal C}}
\newcommand{\cld}{{\cal D}}

\newcommand{\clg}{{\cal G}}

\newcommand{\clk}{{\cal K}}
\newcommand{\cll}{{\cal L}}

\newcommand{\clo}{{\cal O}}

\def\a*{{\cal A}_{h,*}}
\def\B{{\cal B}(h)}
\def\B1{{\cal B}_1(h)}
\def\b{{\cal B}^{\rm s.a.}(h)}
\def\b1{{\cal B}^{\rm s.a.}_1(h)}

\newcommand{\ot}{\otimes}

\newcommand{\raro}{\rightarrow}

\def\a*{{\cal A}_{h,*}}
\def\B{{\cal B}(h)}
\def\B1{{\cal B}_1(h)}
\def\b{{\cal B}^{\rm s.a.}(h)}
\def\b1{{\cal B}^{\rm s.a.}_1(h)}

\begin{document}

\title[]{Equivariant $\textrm{C}^*$-correspondences and compact quantum group actions on Pimsner algebras}
\author[Bhattacharjee]{Suvrajit Bhattacharjee}
\address{Matematisk institutt, Universitetet i Oslo, P.O. Box 1053, Blindern, 0316 Oslo, Norway}
\email{suvrajib@math.uio.no}

\author[Joardar]{Soumalya Joardar}
\address{Department of Mathematics and Statistics, Indian Institute of Science Education and Research Kolkata, Mohanpur - 741246, West
	Bengal, India}
\email{soumalya@iiserkol.ac.in}

\begin{abstract}
Let $G$ be a compact quantum group. We show that given a $G$-equivariant $\textrm{C}^*$-correspondence $E$, the Pimsner algebra $\mathcal{O}_E$ can be naturally made into a $G$-$\textrm{C}^*$-algebra. We also provide sufficient conditions under which it is guaranteed that a $G$-action on the Pimsner algebra $\mathcal{O}_E$ arises in this way, in a suitable precise sense. When $G$ is of Kac type, a \textup{KMS} state on the Pimsner algebra, arising from a quasi-free dynamics, is $G$-equivariant if and only if the tracial state obtained from restricting it to the coefficient algebra is $G$-equivariant, under a natural condition. We apply these results to the situation when the $\textrm{C}^*$-correspondence is obtained from a finite, directed graph and draw various conclusions on the quantum automorphism groups of such graphs, both in the sense of Banica and Bichon.
\end{abstract}



\maketitle

\section{Introduction}
In his seminal paper \cite{Pimsner}, Pimsner introduced a class of $\mathrm{C}^*$-algebras, now referred to as Pimsner algebras, that simultaneously generalizes crossed products by $\mathbb{Z}$ and Cuntz-Krieger algebras. The starting point of his construction is, in his terminology, a Hilbert bimodule $(E,\phi)$ over a $\mathrm{C}^*$-algebra $A$; this means that $E$ is a right Hilbert $A$-module together with an isometric $*$-homomorphism $\phi : A \rightarrow \mathcal{L}(E)$ from $A$ to the $\mathrm{C}^*$-algebra $\mathcal{L}(E)$ of adjointable operators on $E$, that provides the left $A$-module structure on $E$. When $E$ is full as a right Hilbert $A$-module, Pimsner went on to obtain, in analogy with the Toeplitz extension proof of the Bott periodicity and the celebrated Pimsner-Voiculescu sequence, in topological $\operatorname{K}$-theory, a six-term sequence in $\operatorname{KK}$-theory, relating the $\operatorname{KK}$-groups of the Pimsner algebra to the $\operatorname{KK}$-groups of the $\mathrm{C}^*$-algebra $A$. In fact, an essential ingredient in the proof is the existence of a Toeplitz extension associated to the Pimsner algebra. 

Since their introduction, Pimsner algebras have been the subject of intense study, from various points of view; even more so, after Katsura (\cite{katsura1}) generalized the construction, removing the injectivity condition on $\phi$. Thus Pimsner's construction, as generalized by Katsura, now works for any $(E,\phi)$ consisting of a right Hilbert $A$-module $E$ and a $*$-homomorphism $\phi : A \rightarrow \mathcal{L}(E)$ from $A$ to the $\mathrm{C}^*$-algebra $\mathcal{L}(E)$. Such a pair is now referred to as a $\mathrm{C}^*$-correspondence over $A$. This generalization led Katsura (\cite{katsura}) to further extend the already extensive list of $\mathrm{C}^*$-algebras that fall in this class, by incorporating graph $\mathrm{C}^*$-algebras; moreover, it was shown in \cite{katsura2003} that Pimsner algebras also cover what was called crossed product by Hilbert $\mathrm{C}^*$-bimodules in \cite{Abadie1998}. It is to be noted that such a generalization was already considered in the beautiful paper \cite{muhly}, also providing a plentiful of examples.

Leveraging the flexibility of Pimsner's construction, many structural properties of the Pimsner algebra $\mathcal{O}_E$ may be studied through a sound grip on the algebra $A$ and the Hilbert $A$-module $E$. Adapting such a viewpoint, \cites{katideal,KPW} studied the ideal structure of the Pimsner algebras. Various approximation properties are studied, for example, in \cite{sz2010}. As already mentioned above, \cites{Abadie1998,Abadie2009} studies Morita equivalence of such algebras. Continuing along the lines initiated by Pimsner and Katsura, \cite{schaf} studies $\operatorname{K}$-theory of crossed products of Pimsner algebras. Crossed products as well as (co)actions of groups are also the subject of study in \cites{haong,kalis}, which we shall return to in a moment, as they form the main theme of the present article. \cite{muhly} exploits dilation theory and views Pimsner algebras as the $\mathrm{C}^*$-envelopes of the tensor algebra of the correspondence $(E,\phi)$. \cite{Neshveyev} provides a detailed study of \textup{KMS} states (and weights) on the Pimsner algebras, a topic which we will again return to in a moment. Providing a connection with the theory of quantum principal bundles, \cite{Gysin} exhibits a class of natural examples arising from $q$-deformations as Pimsner algebras.  Let us mention also the recent \cite{rrs}, that makes contact with Connes' program \cite{connes}, studying Poincar\'e duality of Pimsner algebras. Finally, the recent preprint \cite{karen} studies correspondences over commutative algebras and associated Pimsner algebras from the point of view of Elliott's program.

At the end of his paper, Pimsner remarks (\cite{Pimsner}*{Remark 4.10}), that
all his constructions are equivariant, under an action of a locally compact,
second countable group. The action of the group on the
$\mathrm{C}^*$-correspondence is to be taken in the sense of Kasparov
(\cite{kasparov}). The details of the remark appear in \cite{haong} where the
authors consider actions of  amenable locally compact groups and show that for
such groups, an equivariant $\mathrm{C}^*$-correspondence induces a natural
action on the Pimsner algebra. Moreover, the crossed product can be identified
as the Pimsner algebra of the crossed product
$\mathrm{C}^*$-correspondence. Continuing along this line, the authors of \cite{kalis} consider coactions of groups on $\mathrm{C}^*$-correspondences and
prove a similar result to that of \cite{haong}; see also \cite{BKQR2015}. The
desire to extend these results to the quantum setting, i.e., in the situation
where we have a quantum group instead of a group, is one of the major
motivations of the present article; and this brings us to the next paragraph.  

Introduced by Woronowicz in his seminal paper \cite{Woro}, compact quantum groups are now well established in Connes' approach to noncommutative geometry. The landmark discovery of $\textup{SU}_q(2)$ by Woronowicz together with the dream of making contact with Connes' enterprise, resulted, following Wang's pioneering work on quantum symmetries of finite spaces (\cite{Wang}), in several constructions and insights. Let us mention, albeit incompletely, the work of 
\begin{itemize}
    \item Banica, Bichon and collaborators on quantum symmetries of discrete structures, see \cites{Ban,banicametric,Bichon};
    \item Goswami, Bhowmick and collaborators on quantum isometries of spectral triples, see \cites{goswamiadv,GJ2018,bg2019,BG2009a};
    \item Banica, Skalski and collaborators on quantum symmetries of $\mathrm{C}^*$-algebras equipped with orthogonal filtrations, see \cites{bs2013,bmrs2019};
    \item and more recently, Goswami and collaborators on quantum symmetries of subfactors, see \cites{bcg2022}. 
\end{itemize}
The study of quantum symmetries of $\mathrm{C}^*$-algebras have also been
rewarding enough. Indeed, for example, it is well-known that there is no ergodic
action of a compact group on the Cuntz-algebra $\clo_{n}$; however,
$\mathcal{O}_n$ admits an ergodic action of a compact quantum group, namely, the
(quasi-free action of the) free unitary quantum group, turning $\clo_{n}$ into a
quantum homogeneous space. Similar richness of quantum symmetries has been
observed in other contexts as well. For example, compact quantum groups have
been found to preserve fewer \textup{KMS} states on certain graph
$\mathrm{C}^*$-algebras as opposed to compact group actions
(\cite{Joardar1}). As a necessarily incomplete list of references for the reader
interested in this direction, we mention \cites{G2014,GW2016,K2017,P1997}.

Keeping in mind the richness of the two camps - Pimsner algebras at one hand and actions of compact quantum groups on the other, we combine the two in the present article. Thus we study compact quantum group actions on Pimsner algebras, the underlying philosophy being the same as mentioned above, i.e., studying such actions through actions on the $\mathrm{C}^*$-correspondence. To carry out this program, however, we would need a notion of equivariant $\mathrm{C}^*$-correspondences under the action of a  compact quantum group. This is based on the fundamental work of Baaj and Skandalis \cite{BS}, where the authors generalize Kasparov's equivariant $\operatorname{KK}$-theory (\cite{kasparov}) to the setting where there is no group anymore but a Hopf $\mathrm{C}^*$-algebra. Having a notion of equivariant $\mathrm{C}^*$-correspondences at hand, our first theorem reads as follows.

\begin{theorem}
\label{thm:maintheoremintro}
Let $G$ be a compact quantum group, $(A,\alpha)$ be a unital $G$-$\mathrm{C}^*$-algebra and  $(E,\phi,\lambda)$ be a $G$-equivariant $\mathrm{C}^*$-correspondence over the $G$-$\mathrm{C}^*$-algebra $(A,\alpha)$. Assume further that the Hilbert $A$-module $E$ is finitely generated and projective.	Then there is a unique unital $*$-homomorphism \[\omega : \mathcal{O}_E \rightarrow \mathcal{O}_E \otimes \mathrm{C}(G)\] such that \[\omega \circ k_E=(k_E \otimes \mathrm{id}_{\mathrm{C}(G)}) \circ \lambda, \text{ and, } \omega \circ k_A=(k_A \otimes \mathrm{id}_{\mathrm{C}(G)})\circ \alpha.\]Moreover, the pair $(\mathcal{O}_E,\omega)$ is a $G$-$\mathrm{C}^*$-algebra. Here, $\mathcal{O}_E$ denotes the Pimsner algebra associated to $(E,\phi)$; $(k_E,k_A)$ is the defining universal covariant representation of $\mathcal{O}_E$; $\alpha$ is the $G$-action on $A$ and $\lambda$ is the $G$-action on $E$.
\end{theorem}

The above theorem also leads one, naturally, to seek for a possible converse to the theorem. However, to identify the precise formulation of a converse, if at all possible, requires some work. To explain in more detail, let us make the following definition. 

\begin{definition}
\label{def:liftedactionintro}
Let $A$ be a unital $\mathrm{C}^*$-algebra, $(E,\phi)$ be a $\mathrm{C}^*$-correspondence over $A$ (where $E$ is assumed to be finitely generated and projective), and $G$ be a compact quantum group. An action $\rho : \mathcal{O}_E \rightarrow \mathcal{O}_E \otimes \mathrm{C}(G)$ of $G$ on the Pimsner algebra $\mathcal{O}_E$ is said to be a lift if there are $G$-actions $\alpha$ and $\lambda$ on $A$ and on $E$, respectively, such that the following are satisfied.
	\begin{itemize}
	\item $(A,\alpha)$ is a $G$-$\mathrm{C}^*$-algebra;
	\item $(E,\phi,\lambda)$ is a $G$-equivariant $\mathrm{C}^*$-correspondence over the $G$-$\mathrm{C}^*$-algebra $(A,\alpha)$;
	\item $\rho$ coincides with $\omega$ as in Theorem \ref{thm:maintheoremintro}.
\end{itemize}
\end{definition}

Now we can state the question in precise terms.

\begin{question}
\label{que:lift}
Given an action $\rho$ of a compact quantum group $G$ on the Pimsner algebra $\mathcal{O}_E$, is $\rho$ always a lift?
\end{question}

The answer to the above question is no, however, and counter-examples exist even for group actions. And this leads to our next theorem.

\begin{theorem}
\label{thm:nonlinearactionintro} 
The action $\rho$ of $\mathbb{T}^n$ on $\clo_{n}$ given by
\[
\rho : \mathcal{O}_n \rightarrow \mathcal{O}_n \otimes \mathrm{C}(\mathbb{T}^n), \quad \rho(S_{i})=(S_{i} \otimes 1_{\mathrm{C}(\mathbb{T}^n)})u,
\]is not a lift. Here, $\mathcal{O}_n$ is the Cuntz algebra on $n$-generators, the generators being $S_i$, $i=1,\dots,n$ and $u$ is the element $\sum_{k=1}^{n}S_{k}S_{k}^*\ot z_{k} \in \clo_{n}\ot \mathrm{C}(\mathbb{T}^{n})$.
\end{theorem}

Nevertheless, we are able to answer Question \ref{que:lift} positively if we restrict ourselves to the class of Pimsner algebras that are considered in \cite{Gysin}, i.e., quantum principal $\mathbb{T}$-bundles, as stated in the following theorem.

\begin{theorem}
\label{thm:sufficiencyintro}
Let $(A,\gamma)$ be a unital $\mathbb{T}$-$\mathrm{C}^*$-algebra such that
\begin{itemize}
	\item the $\mathbb{T}$-action $\gamma$ is principal;
	\item the fixed point algebra $A(0)$ is separable;
	\item the spectral subspaces $A(1)$ and $A(-1)$ are full over $A(0)$,
\end{itemize}so that there is an isomorphism $\mathcal{O}_{A(1)} \cong A$. Let $G$ be a compact quantum group and $\rho : A \rightarrow A \otimes \mathrm{C}(G)$ be a gauge-equivariant $G$-action on $A$ in the sense that for all $z \in \mathbb{T}$, 
\[(\gamma_z \otimes \mathrm{id}_{\mathrm{C}(G)})\circ \rho=\rho \circ \gamma_z.
\] Then $\rho$ is a lift in the sense of Definition \textup{\ref{def:liftedactionintro}}.
\end{theorem}

As mentioned previously, compact quantum group actions preserve fewer \textup{KMS} states on certain graph $\mathrm{C}^*$-algebras and it is thus natural to investigate what happens when we have such an action of a compact quantum group on Pimsner algebras as in Theorem \ref{thm:maintheoremintro}. In \cite{Neshveyev}, the authors show that for a quasi-free dynamics on the Pimsner algebra $\mathcal{O}_E$ induced by a continuous one-parameter group of unitary isometries of the $\mathrm{C}^*$-correspondence $(E,\phi)$, \textup{KMS} states on $\mathcal{O}_E$ are characterized by traces on the $\mathrm{C}^*$-algebra $A$. It is thus natural to hope that if the one-parameter group of isometries of the $\mathrm{C}^*$-correspondence is, in some natural way, compatible with the $G$-structure on the $\mathrm{C}^*$-correspondence then $G$-equivariance of \textup{KMS} states on the Pimsner algebra may also be characterized by $G$-equivariance of the corresponding tracial states on $A$. The necessary compatibility turns out to be the $G$-equivariance of the generator of the one-parameter group of isometries. Let us now state our next theorem which sums up this paragraph.

\begin{theorem}
\label{thm:kmsequivintro}
Let $G$ be a compact quantum group of Kac type, $(A,\alpha)$ be a unital $G$-$\mathrm{C}^*$-algebra, $(E,\phi,\lambda)$ be a $G$-equivariant $\mathrm{C}^*$-correspondence over the $G$-$\mathrm{C}^*$-algebra $(A,\alpha)$ and $\omega$ be the $G$-action on $\mathcal{O}_E$, as obtained in Theorem \textup{\ref{thm:maintheoremintro}} where the Hilbert $A$-module $E$ is finitely generated and projective. Let $\delta$ be the quasi-free dynamics induced by the module dynamics $U$ satisfying the conditions as laid out in \cite{Neshveyev}. Let $U$ be $G$-equivariant, i.e., for all $t \in \mathbb{R}$, \[(U_t \otimes \mathrm{id}_{\mathrm{C}(G)})\circ \lambda=\lambda \circ U_t.\] Let $\varphi$ be a $(\delta,\beta)$-\textup{KMS} state (where $\beta \in (0,\infty)$) on $\mathcal{O}_E$ and $\tau=\varphi \circ k_A$ be the tracial state on $A$ as mentioned above. Then $\varphi$ is $G$-equivariant if and only if $\tau$ is $G$-equivariant.
\end{theorem}

We next apply the above general results to the concrete situation where the $\mathrm{C}^*$-correspondence arises from a finite, directed graph. In that case, the graph $\mathrm{C}^*$-algebra coincides with the Pimsner algebra, allowing us to apply the above results. In particular, we recover the results obtained in \cites{Web,Joardar1,Joardar2}. This also enables us to gain a more concrete understanding of most of the results concerning the interaction between quantum symmetries of graphs and its graph $\mathrm{C}^*$-algebras. For quite a long time, there have been at least two notions of a quantum automorphism group of a finite, simple, directed graph, namely, one due to Banica (\cite{Ban}) and one due to Bichon (\cite{Bichon}). The relationship between these two notions, however, is not so conspicuous; in particular, it is in general difficult to identify the cases when these two notions coincide. In this direction, we have at our disposal the following theorem to offer.

\begin{theorem}
\label{thm:BanBicsameintro}
Let $\clg$ be a finite, simple graph without any source. If either $r$ or $s$ is injective, then $\mathrm{Aut}^{+}_{\textup{Bic}}(\clg)$ is isomorphic to $\mathrm{Aut}^{+}_{\textup{Ban}}(\clg)$.
\end{theorem} 

Let us now briefly discuss the organization of the paper. In Section \ref{sec:cqg}, we briefly recall the necessary background on compact quantum groups and their actions on $\mathrm{C}^*$-algebras and Hilbert $\mathrm{C}^*$-modules. We begin Section \ref{sec:correspondence} with a careful summary of Pimsner's constructions and prove Theorem \ref{thm:maintheoremintro} (Theorem \ref{thm:maintheorem}). It is in this section that we prove Theorem \ref{thm:nonlinearactionintro} (Theorem \ref{thm:nonlinearaction}) and Theorem \ref{thm:sufficiencyintro} (Theorem \ref{thm:sufficiency}) as well. Section \ref{sec:kmsstates} is devoted to recalling some background on \textup{KMS} states from \cite{Neshveyev} and to proving Theorem \ref{thm:kmsequivintro} (Theorem \ref{thm:kmsequiv}). In the remaining three sections, we apply the general results to the question of quantum symmetries of graphs. In Section \ref{sec:applicationsone}, we discuss the bridge that connects the general results of the previous sections with the situation at hand. We also reprove some results on the quantum symmetries of graphs. Section \ref{sec:applicationstwo} is devoted to a detailed study and comparison in the case of a simple, directed graph. It is in this section that we prove Theorem \ref{thm:BanBicsameintro} (Corollary \ref{cor:BanBicsame}). The next and the final section, Section \ref{sec:applicationsthree} is devoted to the case of multigraphs.

To end this Introduction, let us mention that when a first draft of the present article was being written, the preprint \cite{Kim} was brought to our notice. The author proves a similar result to that of ours but goes on to another direction, along the lines of \cites{haong,kalis}. In particular, the author considers coactions of not-necessarily-unital Hopf $\mathrm{C}^*$-algebras on Pimsner algebras and identifies the crossed product as the Pimsner algebra of the crossed product $\mathrm{C}^*$-correspondence. To our surprise, our proof of Theorem \ref{thm:maintheoremintro} is very different to that of \cite{Kim}. The author proves Theorem \ref{thm:maintheoremintro} for a general $\mathrm{C}^*$-correspondence but under an invariance condition of the Katsura ideal. Whereas, we do not require such invariance but restrict ourselves to the case when the Hilbert $\mathrm{C}^*$-module is finitely generated and projective. There are no other overlaps with the results in \cite{Kim}.

\subsection*{Acknowledgments} The first author was supported by the NFR project 300837 ``Quantum Symmetry'' and the Charles University PRIMUS grant \textit{Spectral Noncommutative Geometry of Quantum Flag Manifolds} PRIMUS/21/SCI/026. He also thanks Karen Strung, R\'eamonn \'O Buachalla, Bhishan Jacelon and Alessandro Carotenuto for several discussions on Pimsner algebras. The second author was partially supported by INSPIRE faculty award (DST/INSPIRE/04/2016/002469) given by the D.S.T., Government of India. He also thanks Arnab Mandal for fruitful discussions at various stages of this project.

\subsection*{Notations and Conventions} For an object $X$ in some category, $\mathrm{id}_X$ denotes the identity morphism of $X$. The norm-closed linear span of a subset $S$ of a Banach space is denoted by $[S]$. For a unital $\mathrm{C}^*$-algebra $A$, $1_A$ denotes the unit element in $A$. For a right Hilbert $A$-module $E$, $\mathcal{L}(E)$ denotes the $\mathrm{C}^*$-algebra of adjointable operators on $E$, and $\mathcal{K}(E)$ denotes the closed two-sided ideal of compact operators on $E$. The latter is the norm-closed linear span of $\theta_{\xi,\eta}$, $\xi, \eta \in E$ - where $\theta_{\xi,\eta}$ is the rank-one operator on $E$ given by $\theta_{\xi,\eta}(\zeta)=\xi\langle\eta,\zeta\rangle$, $\zeta \in E$.	We denote the algebraic tensor product by $\odot$. Depending on the context, $\otimes$ denotes the minimal tensor product of two $\mathrm{C}^*$-algebras, or external (also called exterior) tensor product of two Hilbert $\mathrm{C}^*$-modules. However, for the internal tensor product of a Hilbert $A$-module $E$ and a Hilbert $B$-module $F$, where $F$ is endowed with a left action of $A$ via the $*$-homomorphism $\phi : A \rightarrow \mathcal{L}(F)$, is denoted by $E \otimes_{\phi} F$. A reference for the general theory of Hilbert $\mathrm{C}^*$-modules is \cite{Lance}.

\medskip
\textit{All Hilbert $\mathrm{C}^*$-modules considered in this paper are over unital $\mathrm{C}^*$-algebras and are assumed to be full, finitely generated and projective.}

\section{Compact quantum groups and their actions}
\label{sec:cqg}
In this section, we recall the basic definitions from the theory of compact quantum groups and their actions on $\mathrm{C}^*$-algebras. Our reference is \cite{nt2013}; see also \cites{Van,Woro,Wang}.

\begin{definition}
A \textit{compact quantum group} $G$ is a pair $(\mathrm{C}(G),\Delta_G)$ consisting of a unital $\mathrm{C}^*$-algebra $\mathrm{C}(G)$ and a unital $*$-homomorphism $\Delta_G : \mathrm{C}(G) \raro \mathrm{C}(G) \ot \mathrm{C}(G)$ satisfying the following conditions.
\begin{itemize}
	\item $(\mathrm{id}_{\mathrm{C}(G)} \ot \Delta_G) \circ \Delta_G=(\Delta_G \ot \mathrm{id}_{\mathrm{C}(G)}) \circ \Delta_G$ (coassociativity);
	\item $[\Delta_G(\mathrm{C}(G))(1_{\mathrm{C}(G)} \ot \mathrm{C}(G))]=[\Delta_G(\mathrm{C}(G))(\mathrm{C}(G) \ot 1_{\mathrm{C}(G)})]=\mathrm{C}(G) \ot \mathrm{C}(G)$ (bisimplifiability).
\end{itemize}
\end{definition}

Given a compact quantum group $G$, there is a canonical dense Hopf-$*$-algebra $\mathbb{C}[G] \subset \mathrm{C}(G)$ on which an \textit{antipode} $\kappa$ and a \textit{counit} $\varepsilon$ are defined. A \textit{morphism} $f : G_1 \rightarrow G_2$ between two compact quantum groups $G_1$ and $G_2$ is given by a $*$-homomorphism \[f : \mathrm{C}(G_{2}) \raro \mathrm{C}(G_{1}) \text{ such that } (f \ot f)\circ\Delta_{G_2}=\Delta_{G_1}\circ f.\]Such a $*$-homomorphism $f : \mathrm{C}(G_{2}) \raro \mathrm{C}(G_{1})$ is also called a \textit{Hopf $*$-homomorphism} and we will use these two terms interchangeably. 

Let $G$ be a compact quantum group. Then there is a unique state $h$, called the \textit{Haar state}, such that \[(h\ot \mathrm{id}_{\mathrm{C}(G)})\Delta_G(a)=(\mathrm{id}_{\mathrm{C}(G)} \ot h)\Delta_G(a)=h(a)1_{\mathrm{C}(G)} \text { for all } a \in \mathrm{C}(G).\] 

In general, the Haar state need not be tracial but when it is, the compact quantum group $G$ is said to be of \textit{Kac} type.

\begin{definition}
\label{def:Galg}
Let $G$ be a compact quantum group. A \textit{$G$-$\mathrm{C}^*$-algebra} is a pair $(A,\alpha)$ consisting of a unital $\mathrm{C}^*$-algebra $A$ and a unital $*$-homomorphism $\alpha : A \raro A \ot \mathrm{C}(G)$ satisfying the following conditions.
\begin{itemize}
	\item $(\alpha \ot \mathrm{id}_{\mathrm{C}(G)})\circ \alpha=(\mathrm{id}_{A} \ot \Delta_G)\circ \alpha$ (coassociativity);
	\item $[\alpha(A)(1_{A} \ot \mathrm{C}(G))]=A \ot \mathrm{C}(G)$ (Podle\'s condition).
\end{itemize}
\end{definition}

Let $G$ be a compact quantum group and $(A,\alpha)$ be a
$G$-$\mathrm{C}^*$-algebra. One refers to $\alpha$ as the $G$-action on
$A$. There is (\cite{de-commer:action}*{pages 49-50}) a norm-dense $*$-subalgebra $\mathcal{S}(A)$ of $A$, called the \textit{spectral subalgebra} (or the \textit{Podle\'s subalgebra} after \cite{podles}*{Theorem 1.5}), such that $\alpha$ restricts to yield a Hopf-$*$-algebraic coaction $\alpha|_{\mathcal{S}(A)} : \mathcal{S}(A) \rightarrow \mathcal{S}(A) \odot \mathbb{C}[G]$. The $G$-action $\alpha$ is said to be \textit{faithful} if the $*$-algebra generated by the set  \[\{(\omega\ot \mathrm{id}_{\mathrm{C}(G)})\alpha(A) \mid \omega \in A^*\}\] is norm-dense in $\mathrm{C}(G)$. Furthermore, a \textit{$G$-equivariant state} on $A$ is a state $\tau$ on $A$ satisfying \[(\tau \ot \mathrm{id}_{\mathrm{C}(G)})\alpha(a)=\tau(a)1_{\mathrm{C}(G)} \text{ for all } a \in A.\]

\begin{definition}
\label{def:qaut}\cite{Wang}*{Definition 2.3}
Let $A$ be a unital $\mathrm{C}^*$-algebra. A \textit{quantum automorphism group} of $A$ is a pair $(G,\alpha^G)$ consisting of a compact quantum group $G$ and a unital $*$-homomorphism $\alpha^G : A \rightarrow A \otimes \mathrm{C}(G)$ satisfying the following conditions.
\begin{itemize}
	\item The pair $(A,\alpha^G)$ is a $G$-$\mathrm{C}^*$-algebra;
	\item the $G$-action $\alpha^G$ on $A$ is faithful;
	\item if $(G',\alpha^{G'})$ is another pair consisting of a compact quantum group $G'$ and a unital $*$-homomorphism $\alpha^{G'} : A \rightarrow A \otimes \mathrm{C}(G')$ such that $(A,\alpha^{G'})$ is a $G'$-$\mathrm{C}^*$-algebra and $\alpha^{G'}$ is faithful then there is a unique morphism $f : G' \rightarrow G$ such that $(\mathrm{id}_A \otimes f)\circ\alpha^G=\alpha^{G'}$.
\end{itemize}
\end{definition}

\begin{remark}
In general, a quantum automorphism group may fail to exist. To ensure existence, one generally assumes that the considered actions preserve some fixed state on the $\mathrm{C}^*$-algebra. We refrain from going into further details, instead refer the interested reader to \cite{Wang}.
\end{remark}

\begin{example}\cite{Wang}*{Theorem 3.1}
Let $X_{n}$ be the space consisting of $n$ points. The quantum automorphism group of the $\mathrm{C}^*$-algebra $\mathrm{C}(X_{n})$ is the \textit{quantum permutation group} together with the permutation action, denoted by $(\textup{S}_{n}^{+},\alpha^{X_n})$. As a $\mathrm{C}^*$-algebra, $\mathrm{C}(\textup{S}_{n}^{+})$ is the universal $\mathrm{C}^*$ algebra generated by $q_{ij}$, $i,j=1,\dots,n$, satisfying the following relations
	\[
	q_{ij}^{2}=q_{ij}, \ q^*_{ij}=q_{ij}, \ \sum_{j=1}^{n}q_{ij}=\sum_{i=1}^{n}q_{ij}=1.
	\]
The comultiplication on the generators is given by $\Delta_{S^+_n}(q_{ij})=\sum_{k=1}^{n}q_{ik}\ot q_{kj}$, $i,j=1,\dots,n$. The action $\alpha^{X_n}$ is given by $\alpha^{X_n}(e_j)=\sum_{i=1}^ne_i \otimes q_{ij}$, $j=1,\dots,n$, where $e_j$ is the indicator function at the point $j$.
\end{example}

We end this section with one more definition and a remark.

\begin{definition}\cite{BS}*{Definition 2.2}
\label{def:equivariantmodule}
Let $G$ be a compact quantum group and $(B,\beta)$ be a $G$-$\mathrm{C}^*$-algebra. A \textit{$G$-equivariant Hilbert $B$-module} is a pair $(E,\lambda)$ consisting of a right Hilbert $B$-module $E$ and a linear map $\lambda : E \rightarrow E \otimes \mathrm{C}(G)$ satisfying the following.
\begin{itemize}
	\item For all $b \in B$ and $\xi,\eta \in E$, $\lambda(\xi b)=\lambda(\xi)\beta(b)$ and $\langle \lambda(\xi),\lambda(\eta) \rangle=\beta(\langle \xi,\eta \rangle)$;
	\item $(\mathrm{id}_E \otimes \Delta_{G})\circ \lambda=(\lambda \otimes \mathrm{id}_{\mathrm{C}(G)})\circ \lambda$ (coassociativity);
	\item $[\lambda(E)(1_B \otimes \mathrm{C}(G))]=E \otimes \mathrm{C}(G)$ (Podle\'s condition).
\end{itemize}
\end{definition}
	
\begin{remark}
Let $G$ be a compact quantum group, $(B,\beta)$ be a $G$-$\mathrm{C}^*$-algebra and $(E,\lambda)$ be a $G$-equivariant Hilbert $B$-module. One refers to $\lambda$ as the $G$-action on $E$. There is (\cite{voigtbc}*{page 1878}) a dense subspace $\mathcal{S}(E)$ of $E$, called the \textit{spectral submodule}, such that $\lambda$ restricts to yield a coaction $\lambda|_{\mathcal{S}(E)} : \mathcal{S}(E) \rightarrow \mathcal{S}(E) \odot \mathbb{C}[G]$. Moreover, $\mathcal{S}(E)$ is naturally a right $\mathcal{S}(B)$-module and the scalar product of $E$ restricts to an $\mathcal{S}(B)$-valued scalar product on $\mathcal{S}(E)$, making $\mathcal{S}(E)$ into a pre-Hilbert $\mathcal{S}(B)$-module. Furthermore, one constructs (\cite{BS}*{Proposition 2.4}) a unitary operator $V_{\lambda} : E \otimes_{\beta} (B \otimes \mathrm{C}(G)) \rightarrow E \otimes \mathrm{C}(G)$ by \[V_{\lambda}(\xi \otimes x)=\lambda(\xi)x,\] for $\xi \in E$ and $x \in B \otimes \mathrm{C}(G)$. This unitary in turn determines an action $\mathrm{ad}_{\lambda} : \mathcal{K}(E) \rightarrow \mathcal{K}(E) \otimes \mathrm{C}(G)$ on $\mathcal{K}(E)$ given by \[\mathrm{ad}_{\lambda}(T)=V_{\lambda}(T \otimes \mathrm{id})V^*_{\lambda},\] for $T \in \mathcal{K}(E)$, making $(\mathcal{K}(E),\mathrm{ad}_{\lambda})$ into a $G$-$\mathrm{C}^*$-algebra. Setting $T=\theta_{\xi,\eta}$ for $\xi,\eta \in E$, one obtains $\mathrm{ad}_{\lambda}(\theta_{\xi,\eta})=\lambda(\xi)\lambda(\eta)^*=\theta_{\lambda(\xi),\lambda(\eta)}$.
\end{remark}
	
\section{\texorpdfstring{$\mathrm{C}^*$-}{}correspondences and Pimsner algebras}
\label{sec:correspondence}
In this section, after gathering some preliminaries on $\mathrm{C}^*$-correspondences, and on Pimsner algebras, we prove that for a compact quantum group $G$, a $G$-equivariant $\mathrm{C}^*$-correspondence naturally gives rise to a $G$-$\mathrm{C}^*$-algebra structure on the Pimsner algebra. We provide necessary and sufficient conditions under which it is guaranteed that a $G$-$\mathrm{C}^*$-algebra structure on the Pimsner algebra arises in this way. For the basic definitions, our reference is \cite{katsura1}; see also \cites{Pimsner,KPW}. 

\begin{definition}{}
\label{def:cor}
Let $A$ be a unital $\mathrm{C}^*$-algebra. A \textit{$\mathrm{C}^*$-correspondence} over $A$ is a pair $(E,\phi)$, consisting of a right Hilbert $A$-module $E$ and a unital $*$-homomorphism $\phi:A\raro \mathcal{L}(E)$ from the $\mathrm{C}^*$-algebra $A$ to the $\mathrm{C}^*$-algebra $\mathcal{L}(E)$ of adjointable operators on $E$.
\end{definition} 

\begin{remark}
We remark that it is not customary to assume $\phi$ to be unital. But it is easy to see that we lose no generality assuming this. For the justification, we urge the reader to see Remark \ref{rem:phiunital} below.
\end{remark}

\begin{definition}
\label{def:rep}
Let $A$ be a unital $\mathrm{C}^*$-algebra and $(E,\phi)$ be a $\mathrm{C}^*$-correspondence over $A$. A {\it representation} of the $\mathrm{C}^*$-correspondence $(E,\phi)$ on a unital $\mathrm{C}^*$-algebra $B$ is a pair $(t,\pi)$, consisting of a linear map $t : E \raro B$ and a unital $*$-homomorphism $\pi : A \raro B$ such that
\begin{itemize}
	\item $t(\xi)^*t(\eta)=\pi(\langle\xi,\eta\rangle)$ for all $\xi,\eta \in E$;
	\item $t(\phi(a)\xi)=\pi(a)t(\xi)$ for all $a \in A$ and all $\xi \in E$.
\end{itemize}
\end{definition}
It is a fact that for such a representation $(\pi,t)$ of the $\mathrm{C}^*$-correspondence $(E,\phi)$ on the $\mathrm{C}^*$-algebra $B$, $t(\xi a)=t(\xi)\pi(a)$ for all $\xi \in E$, and $a\in A$. Furthermore, there is a well-defined unital $*$-homomorphism $\psi_{t} : \clk(E) \raro B$ such that $\psi_{t}(\theta_{\xi,\eta})=t(\xi)t(\eta)^*$, for $\xi, \eta \in E$. Moreover, we have $\pi(a)\psi_t(k)=\psi_t(\phi(a)k)$ and $\psi_{t}(k)t(\xi)=t(k\xi)$ for all $a \in A$, $k \in \clk(E)$, and $\xi \in E$. 

\begin{definition}{}
\label{def:covariant}
Let $A$ be a unital $\mathrm{C}^*$-algebra and $(E,\phi)$ be a $\mathrm{C}^*$-correspondence over $A$. A representation $(t,\pi)$ of $(E,\phi)$ on a unital $\mathrm{C}^*$-algebra $B$ is said to be {\it covariant} if $\psi_{t}(\phi(a))=\pi(a)$ for all $a \in (\ker \phi)^{\perp}$, where 
\[
	(\ker \phi)^{\perp}=\{a \in A \mid ab=0 \text{ for all } b \in \ker(\phi)\}.
\]
\end{definition}

\begin{remark}
We remark that by our standing assumption, $E$ is finitely generated and therefore $\mathcal{K}(E)=\mathcal{L}(E)$, which also implies that $\phi^{-1}(\mathcal{K}(E))$ coincides with the whole of $A$. Therefore, the so-called Katsura ideal $J_E$, which is defined to be 
\[J_E=\phi^{-1}(\mathcal{K}(E)) \cap (\ker \phi)^{\perp},\] is nothing but $(\ker \phi)^{\perp}$. Moreover, it is proved in \cite{KPW} that one can restrict to covariant representations $(t,\pi)$ on unital $\mathrm{C}^*$-algebras $B$ with $\pi$ unital, to define the Pimsner algebra $\mathcal{O}_E$, assuming $\phi$ to be isometric, which is taken care of by the appearance of $(\ker \phi)^{\perp}$ in the definition of covariance above. 
\end{remark}

\begin{definition}
\label{def:pimsmeralgebra}
Let $A$ be a unital $\mathrm{C}^*$-algebra and $(E,\phi)$ be a $\mathrm{C}^*$-correspondence over $A$. The {\it Pimsner algebra} associated to the $\mathrm{C}^*$-correspondence $(E,\phi)$ is the unital $\mathrm{C}^*$-algebra $\mathcal{O}_E$ satisfying the following.
\begin{itemize}
	\item There is a covariant representation $(k_E,k_A)$ of the $\mathrm{C}^*$-correspondence $(E,\phi)$ on the $\mathrm{C}^*$-algebra $\clo_{E}$, called the universal covariant representation;
	\item for any covariant representation $(t,\pi)$ of the $\mathrm{C}^*$-correspondence $(E,\phi)$ on a unital $\mathrm{C}^*$-algebra $B$, there is a unique unital $*$-homomorphism $t \times \pi : \clo_{E} \raro B$, called the integrated form of $(t,\pi)$, such that,
	\[(t \times \pi)\circ k_E=t, \ (t \times \pi)\circ k_A=\pi.\]
\end{itemize}
\end{definition} 

\begin{remark}
\label{rem:phiunital}
As mentioned above, we lose no generality by assuming $\phi$ to be unital in our definition of a $\mathrm{C}^*$-correspondence. Indeed, letting the projection $\phi(1)$ to be $p$, for a representation $(t,\pi)$, we have \[0=t(\phi(1)(1-p)(\xi))=\pi(1)t((1-p)\xi)=t((1-p)\xi),\]implying that $t$ vanishes identically on the submodule $(1-p)E$. In particular, for the universal representation $(k_E,k_A)$, $k_E$ vanishes identically on $(1-p)E$ and so we may restrict ourselves to the submodule $pE$. But $\phi(1)$ is the identity operator on this submodule. Therefore, we can and do assume $\phi$ to be unital.
\end{remark}

By the universality, for each $z\in \mathbb{T}$, there is a unital $*$-automorphism $\gamma_{z} : \clo_{E} \raro \clo_{E}$ such that $\gamma_{z}(k_A(a))=k_A(a)$ for all $a \in A$ and $\gamma_{z}(k_E(\xi))=zk_E(\xi)$ for all $\xi \in E$. One observes that $z \mapsto \gamma_z$ is strongly continuous and the resulting action of $\mathbb{T}$, $\gamma : \mathcal{O}_E \rightarrow \mathcal{O}_E \otimes \mathrm{C}(\mathbb{T})$ is called the \textit{gauge action}.

\begin{definition} \cite{BS}*{Definition 2.9}
\label{def:equivcorr} 
Let $G$ be a compact quantum group and $(A,\alpha)$ be a $G$-$\mathrm{C}^*$-algebra. A \textit{$G$-equivariant $\mathrm{C}^*$-correspondence} over the $G$-$\mathrm{C}^*$-algebra $(A,\alpha)$ is a triple $(E,\phi,\lambda)$, consisting of a right Hilbert $A$-module $E$, a unital $*$-homomorphism $\phi : A \rightarrow \mathcal{L}(E)$ and a linear map $\lambda : E \rightarrow E \otimes \mathrm{C}(G)$ satisfying the following.
\begin{itemize}
	\item The pair $(E,\phi)$ is a $\mathrm{C}^*$-correspondence over $A$;
	\item the pair $(E,\lambda)$ is a $G$-equivariant Hilbert $A$-module over the $G$-$\mathrm{C}^*$-algebra $(A,\alpha)$ (Definition \ref{def:equivariantmodule});
	\item The $*$-homomorphism $\phi : A \rightarrow \mathcal{L}(E)$ is $G$-equivariant, i.e., \[(\phi \otimes \mathrm{id}_{\mathrm{C}(G)})\alpha(a)=V_{\lambda}(\phi(a) \otimes \mathrm{id})V^*_{\lambda}, \text{ for } a \in A.\]
\end{itemize}
\end{definition} 
	
By \cite{BS}*{page 693}, the last condition is equivalent to
\begin{itemize}
	\item $\lambda(\phi(a)\xi)=(\phi \otimes \mathrm{id}_{\mathrm{C}(G)})(\alpha(a))\lambda(\xi)$ for $\xi \in E$ and $a \in A$. 
\end{itemize}
	
\begin{theorem}
\label{thm:maintheorem}
Let $G$ be a compact quantum group, $(A,\alpha)$ be a $G$-$\mathrm{C}^*$-algebra and  $(E,\phi,\lambda)$ be a $G$-equivariant $\mathrm{C}^*$-correspondence over the $G$-$\mathrm{C}^*$-algebra $(A,\alpha)$. Then there is a unique unital $*$-homomorphism \[\omega : \mathcal{O}_E \rightarrow \mathcal{O}_E \otimes \mathrm{C}(G)\] such that \[\omega \circ k_E=(k_E \otimes \mathrm{id}_{\mathrm{C}(G)}) \circ \lambda, \text{ and, } \omega \circ k_A=(k_A \otimes \mathrm{id}_{\mathrm{C}(G)})\circ \alpha.\] Furthermore, the pair $(\mathcal{O}_E,\omega)$ satisfies the conditions in Definition \textup{\ref{def:Galg}}, making $(\mathcal{O}_E,\omega)$ into a $G$-$\mathrm{C}^*$-algebra.
\end{theorem} 
	
\begin{proof}
We define $t : E \raro \clo_{E} \ot \mathrm{C}(G)$ and $\pi : A \raro \clo_{E}\ot \mathrm{C}(G)$ by the following:
\[t=(k_E \otimes \mathrm{id}_{\mathrm{C}(G)}) \circ \lambda, \text{ and, } \pi=(k_A \otimes \mathrm{id}_{\mathrm{C}(G)})\circ \alpha.\]
We claim that $(t,\pi)$ is a covariant representation of the $\mathrm{C}^*$-correspondence $(E,\phi)$ over $A$. Now for $\xi,\eta \in E$, we have
\[
\begin{aligned}
t(\xi)^*t(\eta)={}&((k_E \otimes \mathrm{id}_{\mathrm{C}(G)}) \lambda(\xi))^*((k_E \otimes \mathrm{id}_{\mathrm{C}(G)}) \lambda(\eta))\\
={}&(k_A \otimes \mathrm{id}_{\mathrm{C}(G)})(\langle \lambda(\xi),\lambda(\eta) \rangle)\\
={}&(k_A \otimes \mathrm{id}_{\mathrm{C}(G)})(\alpha(\langle \xi,\eta \rangle))\\
={}&\pi(\langle \xi,\eta \rangle),
\end{aligned}
\]where the second equality is by Definition \ref{def:rep}; and the third equality is by Definition \ref{def:equivariantmodule}. Again, for $a \in A$ and $\xi \in E$, we have
\[
\begin{aligned}
	t(\phi(a)\xi)={}&(k_E \otimes \mathrm{id}_{\mathrm{C}(G)}) \lambda(\phi(a).\xi)\\
	={}&(k_E \otimes \mathrm{id}_{\mathrm{C}(G)})(\phi \otimes \mathrm{id}_{\mathrm{C}(G)})(\alpha(a))\lambda(\xi)\\
	={}&(k_A \otimes \mathrm{id}_{\mathrm{C}(G)})(\alpha(a))(k_E \otimes \mathrm{id}_{\mathrm{C}(G)})\lambda(\xi)\\
	={}&\pi(a)t(\xi),
\end{aligned}
\]where the second equality is by Definition \ref{def:equivariantmodule}; and the third equality is by Definition \ref{def:pimsmeralgebra}. Thus $(t,\pi)$ is a representation of the $\mathrm{C}^*$-correspondence $(E,\phi)$ over $A$. We recall that there is a unital $*$-homomorphism $\psi_t : \clk(E) \raro \clo_{E} \otimes \mathrm{C}(G)$ given by $\psi_t(\theta_{\xi,\eta})=t(\xi)t(\eta)^{\ast}$ such that $\psi_t(k)t(\xi)=t(k\xi)$ for all $\xi \in E$ and $k \in \mathcal{K}(E)$; in particular, we have $\psi_t(\phi(a))t(\xi)=t(\phi(a)\xi)=\pi(a)t(\xi)$ for all $\xi \in E$ and all $a \in A$. Since $[\lambda(E)(1_A \otimes \mathrm{C}(G))]=E \otimes \mathrm{C}(G)$, using the linearity and the continuity of the maps involved, we get for all $\xi \in E$ and $a \in A$,
\[
	\psi_t(\phi(a))(k_E(\xi) \otimes 1_{\mathrm{C}(G)})=\pi(a)(k_E(\xi)\otimes 1_{\mathrm{C}(G)}).
\]Consequently for any $\xi,\eta \in E$, and $a \in A$, we have
\[
	\psi_t(\phi(a))(k_E(\xi)k_E(\eta)^*\otimes 1_{\mathrm{C}(G)})=\pi(a)(k_E(\xi)k_E(\eta)^*\otimes 1_{\mathrm{C}(G)}),
\]
i.e., \[\psi_t(\phi(a))(\psi_{k_E}(\theta_{\xi,\eta})\otimes 1_{\mathrm{C}(G)})=\pi(a)(\psi_{k_E}(\theta_{\xi,\eta})\otimes 1_{\mathrm{C}(G)}).\] So again by linearity and continuity of the maps involved, we get
\[
	\psi_t(\phi(a))(\psi_{k_E}(k)\otimes 1_{\mathrm{C}(G)})=\pi(a)(\psi_{k_E}(k)\otimes 1_{\mathrm{C}(G)}),
\] for all $k \in \mathcal{K}(E)$. As the $\mathrm{C}^*$-algebra $A$ is unital and $E$ is finitely generated, $\clk(E)=\cll(E)$ and since $\psi_{t}$ is a unital $*$-homomorphism, we obtain, by plugging $k=\mathrm{id}_E \in \cll(E)$ in the identity above,
\[
	\psi_t(\phi(a))=\pi(a),
\] for all $a \in A$, which in particular, shows that $(t,\pi)$ is a covariant representation of the $\mathrm{C}^*$-correspondence $(E,\phi)$ over $A$ on the unital $\mathrm{C}^*$-algebra $\clo_{E} \otimes \mathrm{C}(G)$. Therefore, by the universality of $\clo_{E}$, we get a necessarily unique $*$-homomorphism \[t \times \pi=\omega : \clo_{E} \raro \clo_{E} \ot \mathrm{C}(G)\] such that \[\omega \circ k_E=(k_E \otimes \mathrm{id}_{\mathrm{C}(G)}) \circ \lambda, \text{ and, } \omega \circ k_A=(k_A \otimes \mathrm{id}_{\mathrm{C}(G)})\circ \alpha,\] completing the first part of the theorem. To see that $(\mathcal{O}_E,\omega)$ is indeed a $G$-$\mathrm{C}^*$-algebra, we start with the coassociativity of $\omega$, i.e., $(\omega \otimes \mathrm{id}_{\mathrm{C}(G)})\circ\omega=(\mathrm{id}_{\mathcal{O}_E} \otimes \Delta_G)\circ\omega$. Again, by universality (and uniqueness), it suffices to show that the two $*$-homomorphisms $(\omega \otimes \mathrm{id}_{\mathrm{C}(G)})\circ\omega$ and $(\mathrm{id}_{\mathcal{O}_E} \otimes \Delta_G)\circ\omega$ agree on the images $k_E(E)$ and $k_A(A)$, i.e., 
\[
(\omega \otimes \mathrm{id}_{\mathrm{C}(G)})\circ \omega \circ k_E=(\mathrm{id}_{\mathcal{O}_E} \otimes \Delta_G)\circ \omega \circ k_E,
\]and
\[
(\omega \otimes \mathrm{id}_{\mathrm{C}(G)})\circ \omega \circ k_A=(\mathrm{id}_{\mathcal{O}_E} \otimes \Delta_G)\circ \omega \circ k_A.
\]However, using \[\omega \circ k_E=(k_E \otimes \mathrm{id}_{\mathrm{C}(G)}) \circ \lambda, \text{ and, } \omega \circ k_A=(k_A \otimes \mathrm{id}_{\mathrm{C}(G)})\circ \alpha,\]respectively, we see that we are reduced to the coassociativity of $\lambda$ and $\alpha$, respectively. Thus $\omega$ is indeed coassociative. For the Podle\'s condition, we consider the set 
\[
S=\{x \in \mathcal{O}_E \mid x \otimes 1_{\mathrm{C}(G)} \in [\omega(\mathcal{O}_E)(1_{\mathcal{O}_E} \otimes \mathrm{C}(G))]\}.	
\]Since $[\lambda(E)(1_A \otimes \mathrm{C}(G))]=E \otimes \mathrm{C}(G)$ and $[\alpha(A)(1_A \otimes \mathrm{C}(G))]=A \otimes \mathrm{C}(G)$, we see that for each $\xi \in E$, $k_E(\xi) \in S$ and for each $a \in A$, $k_A(a) \in S$. Let $x$ and $y$ be in $S$. Then we see that
\[
\begin{aligned}
	xy \otimes 1_{\mathrm{C}(G)}=(x \otimes 1_{\mathrm{C}(G)})(y \otimes 1_{\mathrm{C}(G)}) \in{}& [\omega(\mathcal{O}_E)(1_{\mathcal{O}_E} \otimes \mathrm{C}(G))(y \otimes 1_{\mathrm{C}(G)})]\\
	={}&[\omega(\mathcal{O}_E)(y \otimes 1_{\mathrm{C}(G)})(1_{\mathcal{O}_E} \otimes \mathrm{C}(G))]\\
	\subseteq{}&[\omega(\mathcal{O}_E)\omega(\mathcal{O}_E)(1_{\mathcal{O}_E} \otimes \mathrm{C}(G))(1_{\mathcal{O}_E} \otimes \mathrm{C}(G))]\\
	={}&[\omega(\mathcal{O}_E)(1_{\mathcal{O}_E} \otimes \mathrm{C}(G))],
\end{aligned}	
\]i.e., $S$ is closed under multiplication. But $S$ contains, as shown above, $k_E(E)$ and $k_A(A)$; thus again by universality, $S$ equals $\mathcal{O}_E.$ Therefore $(\mathcal{O}_E,\omega)$ is indeed a $G$-$\mathrm{C}^*$-algebra and this completes the proof.
\end{proof}
	
\begin{remarks}\hfill
\begin{itemize}
	\item We do not need the fact that the Hilbert $\mathrm{C}^*$-module $E$ be full in the above proof. 
	\item When the first draft of the present article was being prepared, the preprint \cite{Kim} was brought to our notice. In \cite{Kim}, the author proves Theorem \ref{thm:maintheorem} in greater generality, under an invariance assumption of the Katsura ideal $J_E$, however. We do not need such an assumption, and as such our proof is an improvement over the proof given in \cite{Kim}. On the other hand, we require the Hilbert $\mathrm{C}^*$-module $E$ to be finitely generated. 
	\item We prove a somewhat stronger statement than the covariance of $(t,\pi)$, in the notations of the above proof, namely, \[\psi_t \circ \phi=\pi,\]hinging heavily on the unitality of $\mathrm{C}(G)$ and of $\mathcal{K}(E)$.
\end{itemize}
\end{remarks}

\begin{definition}
\label{def:liftedaction}
Let $A$ be a unital $\mathrm{C}^*$-algebra, $(E,\phi)$ be a $\mathrm{C}^*$-correspondence over $A$, and $G$ be a compact quantum group. An action $\rho : \mathcal{O}_E \rightarrow \mathcal{O}_E \otimes \mathrm{C}(G)$ of $G$ on the Pimsner algebra $\mathcal{O}_E$ is said to be \textit{a lift} if there are $G$-actions $\alpha$ and $\lambda$ on $A$ and on $E$, respectively, such that the following are satisfied.
\begin{itemize}
	\item $(A,\alpha)$ is a $G$-$\mathrm{C}^*$-algebra;
	\item $(E,\phi,\lambda)$ is a $G$-equivariant $\mathrm{C}^*$-correspondence over the $G$-$\mathrm{C}^*$-algebra $(A,\alpha)$;
	\item $\rho$ coincides with $\omega$ as in Theorem \ref{thm:maintheorem}.
\end{itemize}
\end{definition}
	
Thanks to the previous theorem, any action of a compact quantum group $G$ on a $\mathrm{C}^*$-correspondence $(E,\phi)$ over a $\mathrm{C}^*$-algebra $A$ automatically lifts to an action on the corresponding Pimsner algebra $\mathcal{O}_E$. However, we cannot expect all actions on the Pimsner algebra to be lifts of actions on the underlying $\mathrm{C}^*$-correspondence in the sense of Definition \ref{def:liftedaction}. One may look for a characterization of all the actions on the Pimsner algebra that are lifts of coactions on the underlying $\mathrm{C}^*$-correspondence. In general however, it seems to be a hard problem to obtain such a characterization. We are able to make some progress in this direction. Let us start with a necessary condition.
	
\begin{proposition}
\label{prop:gaugeequivariance}
Let $G$ be a compact quantum group, $(A,\alpha)$ be a $G$-$\mathrm{C}^*$-algebra and  $(E,\phi,\lambda)$ be a $G$-equivariant $\mathrm{C}^*$-correspondence over $A$. Then for all $z \in \mathbb{T}$, 
\[(\gamma_z \otimes \mathrm{id}_{\mathrm{C}(G)})\circ \omega=\omega \circ \gamma_z,
\]where $\omega$ as in Theorem \textup{\ref{thm:maintheorem}}, and $\gamma$ is the gauge action on $\mathcal{O}_E$.
\end{proposition} 

\begin{proof}
We begin by observing that it suffices to show, by universality (and uniqueness), that the two $*$-homomorphisms $(\gamma_z \otimes \mathrm{id}_{\mathrm{C}(G)})\circ \omega$ and $\omega \circ \gamma_z$ agree on $k_E(E)$ and $k_A(A)$, i.e., 
\[
(\gamma_z \otimes \mathrm{id}_{\mathrm{C}(G)})\circ \omega \circ k_E=\omega \circ \gamma_z \circ k_E,
\]and
\[
(\gamma_z \otimes \mathrm{id}_{\mathrm{C}(G)})\circ \omega \circ k_A=\omega \circ \gamma_z \circ k_A.	
\]However, using \[\omega \circ k_E=(k_E \otimes \mathrm{id}_{\mathrm{C}(G)}) \circ \lambda, \text{ and, } \omega \circ k_A=(k_A \otimes \mathrm{id}_{\mathrm{C}(G)})\circ \alpha,\]respectively, and the explicit form of $\gamma$ provided after Definition \ref{def:pimsmeralgebra}, we see that the above two identities are indeed satisfied. This completes the proof.
\end{proof}
	
Thus it makes sense to single out the class of actions on the Pimsner algebra that satisfy Proposition \ref{prop:gaugeequivariance} and allows us to make the following definition.

\begin{definition}
\label{def:gaugeequivariance}
Let $A$ be a unital $\mathrm{C}^*$-algebra, $(E,\phi)$ be a $\mathrm{C}^*$-correspondence over $A$ and $G$ be a compact quantum group. An action $\rho : \mathcal{O}_E \rightarrow \mathcal{O}_E \otimes \mathrm{C}(G)$ of $G$ on the Pimsner algebra $\mathcal{O}_E$ is said to be \textit{gauge-equivariant} if for all $z \in \mathbb{T}$, 
\[(\gamma_z \otimes \mathrm{id}_{\mathrm{C}(G)})\circ \rho=\rho \circ \gamma_z,
\]where $\gamma$ is the gauge action on $\mathcal{O}_E$.
\end{definition}

The gauge-equivariance is not sufficient though. To produce an example where gauge-equivariance of an action on the Pimsner algebra does not imply that the action is a lift of an action on the $\mathrm{C}^*$-correspondence in the sense of Definition \ref{def:liftedaction}, we recall that the Cuntz algebra $\clo_{n}$ may be seen as the Pimsner algebra arising from the $\mathrm{C}^*$-correspondence of its graph. Denoting the generating partial isometries of $\clo_{n}$ by $S_{i}$, for $i=1,\dots,n$, it is clear that for any action $\rho$ of a compact quantum group $G$ to be a lift of an action on the $\mathrm{C}^*$-correspondence coming from the graph, $\rho$ has to be `linear', in the sense that $\rho(S_{j})=\sum_{i=1}^n S_{i} \ot q_{ij}$ for some $q_{ij} \in \mathrm{C}(G)$, $j=1,\dots,n$. With this observation in hand, we shall produce a `non-linear', gauge-equivariant action of the compact (quantum) group $\mathbb{T}^n=(\mathrm{C}(\mathbb{T}^{n}),\Delta_{\mathbb{T}^n})$ on $\clo_{n}$. We denote the generating unitaries of $\mathrm{C}(\mathbb{T}^{n})$ by $z_{1},\dots,z_{n}$. We recall that the comultiplication is given on $z_{i}$ by $\Delta(z_{i})=z_{i} \ot z_{i}$ for all $i=1,\dots,n$. With these notations, we have the following theorem.
	
\begin{theorem}
\label{thm:nonlinearaction} 
The action $\rho$ of $\mathbb{T}^n$ on $\clo_{n}$ given by
\[
\rho(S_{i})=(S_{i} \otimes 1_{\mathrm{C}(\mathbb{T}^n)})u,
\]is `non-linear', and gauge-equivariant, where $u$ is the element $\sum_{k=1}^{n}S_{k}S_{k}^*\ot z_{k} \in \clo_{n}\ot \mathrm{C}(\mathbb{T}^{n})$.
\end{theorem}
	
\begin{proof}  
We begin by observing that $\rho$ so defined is clearly gauge-equivariant and `non-linear'. We have to show that it indeed defines an action of $\mathbb{T}^n$. First, let us check that $\rho$ is a well-defined $*$-homomorphism. 

To begin with, using the unitarity of $z_{i}$, for $i=1,\dots,n$, it is easy to see that $u$ is a unitary element of $\clo_{n}\ot \mathrm{C}(\mathbb{T}^{n})$, and so for all $1 \leq i,j \leq n$, \[\rho(S^*_i)\rho(S_j)=\delta_{ij}(1\ot 1).\] Moreover, \[\sum_{k=1}^{n}\rho(S_k)\rho(S^*_k)=1 \ot 1,\] so that by the universality of $\clo_{n}$, $\rho$ is well-defined. It is also easy to observe that for all $1 \leq i,j \leq n$, \[\rho(S_iS^*_j)=S_iS^*_j \ot 1\] and coassociativity of $\rho$ follows. Again using \[\rho(S_{i}S_{i}^*)=S_{i}S_{i}^* \ot 1\] for all $i=1,\dots,n$, one can see that \[\sum_{k=1}^{n}\rho(S_{k}S_{k}^*)(1 \ot z_{k}^*)=u^*\] and consequently, \[\sum_{k=1}^{n}\rho(S_{i}S_{k}S_{k}^*)(1 \ot z_{k}^*)=(S_{i} \otimes 1_{\mathrm{C}(\mathbb{T}^n)})uu^*=S_{i} \ot 1_{\mathrm{C}(\mathbb{T}^n)}\] for all $i=1,\dots,n$. Thus $S_{i} \ot 1_{\mathrm{C}(\mathbb{T}^n)}$, for all $i=1,\dots,n$, belongs to $[\rho(\clo_{n})(1_{\mathcal{O}_n} \ot \mathrm{C}(\mathbb{T}^{n}))]$ and arguing as in the last part of the proof of Theorem \ref{thm:maintheorem}, the Podle\'s condition follows. This completes the proof.
\end{proof}
	
Now we shall prove that the gauge-equivariance condition as in Definition \ref{def:gaugeequivariance} is also sufficient, provided we restrict ourselves to the class of Pimsner algebras arising from principal $\mathbb{T}$-bundles, as described below (we provide a detailed description as this example covers a lot of other instances). We remark, however, that the special structure of a $\mathbb{T}$-principal bundle is only used to ensure that the Hilbert $\mathrm{C}^*$-module for the $\mathrm{C}^*$-correspondence is full and finitely generated, so that our Theorem \ref{thm:maintheorem} applies.

\begin{example}
\label{exa:principal}\cites{Gysin,Abadie1998,Abadie2009}
We recall that a unital $\mathbb{T}$-$\mathrm{C}^*$-algebra $(A,\gamma)$ (without any chance of confusion, we denote the toral action by the same symbol, $\gamma$, as was used to denote the gauge action) comes with an associated $\mathbb{Z}$-grading defined as follows. We call an element $a \in A$ homogeneous of degree $n \in \mathbb{Z}$ if $\gamma_z(a)=z^na$ for all $z \in \mathbb{T}$ and write $\mathrm{deg}(a)=n$. For each $n \in \mathbb{Z}$, we let $A(n)$ denote the set consisting of homogeneous elements of degree $n$: $A(n)=\{a \in A \mid \mathrm{deg}(a)=n\}$. The collection $\{A(n)\}_{n \in \mathbb{Z}}$ enjoys the following: 
\begin{itemize}
	\item for each $n \in \mathbb{Z}$, $A(n)$ is a closed subspace of $A$;
	\item for $m,n \in \mathbb{Z}$, $A(m)A(n) \subseteq A(m+n)$;
	\item for each $n \in \mathbb{Z}$, $A(n)^*=A(-n)$;
	\item the algebraic direct sum $\bigoplus_{n \in \mathbb{Z}}A(n)$ is norm-dense in $A$.
\end{itemize}
In particular, the fixed point subalgebra $A(0)$ is a $\mathrm{C}^*$-algebra and the first spectral subspace $A(1)$ is a right Hilbert $A(0)$-module, where the right $A(0)$-module structure is given by multiplication from the right within the $\mathrm{C}^*$-algebra $A$, and the $A(0)$-valued inner product is given by \[\langle \xi,\eta \rangle=\xi^*\eta,\] for $\xi,\eta \in A(1)$. We define the $*$-homomorphism $\phi : A(0) \raro \cll(A(1))$ by \[\phi(a)(\xi)=a \xi,\] for $\xi \in A(1)$ and $a \in A(0)$. Yet, the pair $(A(1),\phi)$ is \textit{not} a $\mathrm{C}^*$-correspondence over $A(0)$ in our sense; we still require that $A(1)$ be full and finitely generated over $A(0)$. We recall that the $\mathbb{T}$-action $\gamma$ on the $\mathbb{T}$-$\mathrm{C}^*$-algebra $A$ is said to be principal if the associated $\mathbb{Z}$-grading is strong, i.e., for $m,n \in \mathbb{Z}$, \[A(m)A(n)=A(m+n).\] Assume further, that the $\mathrm{C}^*$-algebra $A(0)$ is separable, that $A(1)$ and $A(-1)$ are full over $A(0)$. Then $(A(1),\phi)$ is a $\mathrm{C}^*$-correspondence over $A(0)$. In this case, the Pimsner algebra $\clo_{A(1)}$ associated to the $\mathrm{C}^*$-correspondence $(A(1),\phi)$ is isomorphic to $A$ and the gauge action coincides with the initial $\mathbb{T}$-action (thus relieving us from any possible confusion with the choice of notation).  
\end{example}
	
\begin{theorem}
\label{thm:sufficiency}
Let $(A,\gamma)$ be a unital $\mathbb{T}$-$\mathrm{C}^*$-algebra such that
\begin{itemize}
	\item the $\mathbb{T}$-action $\gamma$ is principal;
	\item the fixed point algebra $A(0)$ is separable;
	\item the spectral subspaces $A(1)$ and $A(-1)$ are full over $A(0)$,
\end{itemize}so that by Example \textup{\ref{exa:principal}} above, there is an isomorphism $\mathcal{O}_{A(1)} \cong A$. Let $G$ be a compact quantum group and $\rho : A \rightarrow A \otimes \mathrm{C}(G)$ be a gauge-equivariant $G$-action on $A$ in the sense of Definition \textup{\ref{def:gaugeequivariance}}, i.e., for all $z \in \mathbb{T}$, 
\[(\gamma_z \otimes \mathrm{id}_{\mathrm{C}(G)})\circ \rho=\rho \circ \gamma_z.
\] Then $\rho$ is a lift in the sense of Definition \textup{\ref{def:liftedaction}}.
\end{theorem}
	
\begin{proof}
We begin by observing that the gauge-equivariance  of $\rho$ yields 
\[
\rho(A(0)) \subseteq A(0) \otimes \mathrm{C}(G), \text{ and } \rho(A(1)) \subseteq A(1) \otimes \mathrm{C}(G),
\]i.e., $\rho$ preserves the $\mathrm{C}^*$-algebra and the module $A(1)$. Let us denote the restricted actions $\rho|_{A(0)}$ and $\rho|_{A(1)}$ on $A(0)$ and $A(1)$ by $\alpha$ and $\lambda$, respectively. We will show that $(A(0),\alpha)$ is a $G$-$\mathrm{C}^*$-algebra, $(A(1),\phi,\lambda)$ is a $G$-equivariant $\mathrm{C}^*$-correspondence over the $G$-$\mathrm{C}^*$-algebra $(A,\alpha)$ and that $\rho$ coincides with  $\omega$, where $\omega$ as in Theorem \ref{thm:maintheorem}. 

To that end, we first remark that $\alpha$ is indeed a unital $*$-homomorphism and $\lambda$ is indeed a linear map; these follow from the corresponding property of $\rho$. Next, we observe that coassociativity of both $\alpha$ and $\lambda$ follow from that of $\rho$. Now we fix $\xi, \eta \in A(1)$ and $a \in A(0)$. Then
\[
\lambda(\xi a)=\rho(\xi a)=\rho(\xi)\rho(a)=\lambda(\xi)\alpha(a),
\] 
\[
\langle \lambda(\xi),\lambda(\eta) \rangle=\rho(\xi)^*\rho(\eta)=\rho(\xi^*\eta)=\rho(\langle \xi,\eta \rangle)=\alpha(\langle \xi,\eta \rangle),	
\]and
\[
\lambda(\phi(a)\xi)=\lambda(a\xi)=\rho(a\xi)=\rho(a)\rho(\xi)=(\phi \otimes \mathrm{id}_{\mathrm{C}(G)})(\alpha(a))\lambda(\xi),
\]all the equalities being clear from the facts that $\rho$ is a $*$-homomorphism and that $\lambda=\rho|_{A(1)}$ and $\alpha=\rho|_{A(0)}$. Thus we see that indeed condition (1) of Definition \ref{def:Galg}, conditions (1) and (2) of Definition \ref{def:equivariantmodule} and conditions (1) and (3$^{\prime}$) of Definition \ref{def:equivcorr} are satisfied by $\alpha$ and $\lambda$. The proof will be complete, by uniqueness of $\omega$ in Theorem \ref{thm:maintheorem}, provided we could show the Podle\'s conditions for $\alpha$ and $\lambda$. We now proceed to do so.

We recall that there is a conditional expectation $\operatorname{E} :A \raro A$, given by,
\[
\operatorname{E}(a)=\int_{\mathbb{T}}\gamma_{z}(a) \ dz,
\]and that $A(0)$ coincides with $\operatorname{E}(A)$. Since $\rho$ is gauge-equivariant, we have
\[
\rho \circ \operatorname{E}=(\operatorname{E} \otimes \mathrm{id}_{\mathrm{C}(G)})\circ \rho.	
\]Then 
\[
\begin{aligned}
	A(0) \otimes \mathrm{C}(G)={}&\operatorname{E}(A) \otimes \mathrm{C}(G)\\
	={}&(\operatorname{E} \otimes \mathrm{id}_{\mathrm{C}(G)})(A \otimes \mathrm{C}(G))\\
	={}&(\operatorname{E} \otimes \mathrm{id}_{\mathrm{C}(G)})[\rho(A)(1_A \otimes \mathrm{C}(G))]\\
	={}&[\rho(\operatorname{E}(A))(1_A \otimes \mathrm{C}(G))]\\
	={}&[\alpha(A(0))(1_{A(0)} \otimes \mathrm{C}(G))],
\end{aligned}
\]where the first and fifth equalities use the fact that $A(0)=\operatorname{E}(A)$; the third equality uses the Podle\'s condition for $\rho$; the fourth equality uses the identity just above this computation; and finally the fifth equality uses the fact that $\alpha=\rho|_{A(0)}$. Therefore $(A(0),\alpha)$ is a $G$-$\mathrm{C}^*$-algebra.

Now let $\operatorname{P} : A \rightarrow A$ denote the projection onto $A(1)$, obtained from the Banach space decomposition of $A$,
\[
A=A(1) \oplus \overline{\bigoplus_{n\neq 1}A(n)}.
\]Again, since $\rho$ is gauge-equivariant, we have 
\[
\rho \circ \operatorname{P}=(\operatorname{P} \otimes \mathrm{id}_{\mathrm{C}(G)})\circ \rho.		
\]Then 
\[
\begin{aligned}
	A(1) \otimes \mathrm{C}(G)={}&\operatorname{P}(A) \otimes \mathrm{C}(G)\\
	={}&(\operatorname{P} \otimes \mathrm{id}_{\mathrm{C}(G)})(A \otimes \mathrm{C}(G))\\
	={}&(\operatorname{P} \otimes \mathrm{id}_{\mathrm{C}(G)})[\rho(A)(1_A \otimes \mathrm{C}(G))]\\
	={}&[\rho(\operatorname{P}(A))(1_A \otimes \mathrm{C}(G))]\\
	={}&[\lambda(A(1))(1_{A(0)} \otimes \mathrm{C}(G))],
\end{aligned}
\]where the first and fifth equalities use the fact that $A(1)=\operatorname{P}(A)$; the third equality uses the Podle\'s condition for $\rho$; the fourth equality uses the identity just above this computation; and finally the fifth equality uses the fact that $\lambda=\rho|_{A(1)}$. Therefore $(A(1),\phi,\lambda)$ is a $G$-equivariant $\mathrm{C}^*$-correspondence over the $G$-$\mathrm{C}^*$-algebra $(A(0),\alpha)$. Thus as observed above, by the uniqueness of $\omega$ in Theorem \ref{thm:maintheorem}, the proof is now complete.
\end{proof}

\section{KMS states on the Pimsner algebras}
\label{sec:kmsstates}
In this section, after briefly recalling how a \textup{KMS} state with respect to a quasi-free dynamics on a Pimsner algebra induced by a module dynamics in the sense of \cite{Neshveyev} looks like, we provide a necessary and sufficient condition for it to be a $G$-equivariant state, for the action $\omega$ from Theorem \ref{thm:maintheorem}, where $G$ is a compact quantum group of Kac type. A general reference for \textup{KMS} states on a $\mathrm{C}^*$-algebra is \cite{br}. 

\begin{definition}
	Let $A$ be a unital $\mathrm{C}^*$-algebra and $(E,\phi)$ be a $\mathrm{C}^*$-correspondence over $A$. For $n \in \mathbb{N}$, we define a $\mathrm{C}^*$-correspondence $(E^{(n)},\phi_{(n)})$ over $A$ as follows. We set $E^{(0)}=A$, $E^{(1)}=E$ and $E^{(n+1)}=E \otimes_{\phi_{(n)}} E^{(n)}$ for $n \geq 1$. We also define $\phi_{(0)}$ to be the identity of $A$, $\phi_{(1)}=\phi$ and for $n \geq 1$, $\phi_{(n+1)}(a)=\phi(a) \otimes \mathrm{id}_{E^{(n)}}$. 
\end{definition}

We recall that \[E \otimes_{\phi} E=[\{\xi \otimes \eta \mid \xi, \eta \in E\}],\] and $(\xi a) \otimes \eta=\xi \otimes (\phi(a)\eta)$ for $\xi, \eta \in E$ and $a \in A$. More generally, \[E^{(n)}=[\{\xi_1 \otimes \dots \otimes \xi_n \mid \xi_1,\dots, \xi_n \in E\}].\] 
We refer the reader to \cite{Lance}*{Chapter 4} for a detailed discussion on interior tensor product of Hilbert $\mathrm{C}^*$-modules.

\begin{remark}
	We remark that we have omitted $\phi$ in the expression $\xi \otimes \eta$, for the sake of notational convenience and will do so without further comment.
\end{remark}

\begin{proposition}\cite{BS}*{Proposition 2.10}
\label{prop:Gequivtensor}
Let $G$ be a compact quantum group, $(A,\alpha)$ be a $G$-$\mathrm{C}^*$-algebra and  $(E,\phi,\lambda)$ be a $G$-equivariant $\mathrm{C}^*$-correspondence over the $G$-$\mathrm{C}^*$-algebra $(A,\alpha)$. Then there is a linear map \[\lambda_{(2)} : E^{(2)} \rightarrow E^{(2)} \otimes \mathrm{C}(G),\]such that the triple $(E^{(2)},\phi_{(2)},\lambda_{(2)})$ satisfies the conditions of Definition \textup{\ref{def:equivcorr}}, making $(E^{(2)},\phi_{(2)},\lambda_{(2)})$ into a $G$-equivariant $\mathrm{C}^*$-correspondence over the $G$-$\mathrm{C}^*$-algebra $(A,\alpha)$. In Sweedler notation, for $\xi,\eta \in \mathcal{S}(E)$ (the spectral submodule), $\lambda_{(2)}$ is given by \[\lambda_{(2)}(\xi \otimes \eta)=\xi_{(0)} \otimes \eta_{(0)} \otimes \xi_{(1)}\eta_{(1)}.\]
\end{proposition} 

\begin{remark}
We remark that the above proposition uses, together with (and in the notations from) \cite{BS}*{Proposition 2.10}, the following isomorphisms.
\begin{itemize}
	\item Since both $A$ and $\mathrm{C}(G)$ are unital, we have (\cite{BS}*{page 686}),\[\widetilde{M}(A \otimes \mathrm{C}(G)) \cong A \otimes \mathrm{C}(G);\]
	\item since $\phi$ is unital, we have (\cite{BS}*{page 693}),\[(A \otimes \mathrm{C}(G)) \otimes_{\phi \otimes \mathrm{id}_{\mathrm{C}(G)}} (E \otimes \mathrm{C}(G)) \cong E \otimes \mathrm{C}(G);\]
	\item and finally, we have (\cite{BS}*{page 693, proof of Proposition 2.10}), \[(E \otimes_{\phi} E) \otimes \mathrm{C}(G) \cong (E \otimes \mathrm{C}(G)) \otimes_{\phi \otimes \mathrm{id}_{\mathrm{C}(G)}} (E \otimes \mathrm{C}(G)).\]
\end{itemize}
\end{remark}

\begin{corollary}
\label{cor:equivhigher}
Let $G$ be a compact quantum group, $(A,\alpha)$ be a $G$-$\mathrm{C}^*$-algebra and  $(E,\phi,\lambda)$ be a $G$-equivariant $\mathrm{C}^*$-correspondence over the $G$-$\mathrm{C}^*$-algebra $(A,\alpha)$. Then for each $n \geq 1$, there is a linear map \[\lambda_{(n)} : E^{(n)} \rightarrow E^{(n)} \otimes \mathrm{C}(G),\]such that the triple $(E^{(n)},\phi_{(n)},\lambda_{(n)})$ satisfies the conditions of Definition \textup{\ref{def:equivcorr}}, making $(E^{(n)},\phi_{(n)},\lambda_{(n)})$ into a $G$-equivariant $\mathrm{C}^*$-correspondence over the $G$-$\mathrm{C}^*$-algebra $(A,\alpha)$. In Sweedler notation, for $\xi_1,\dots,\xi_n \in \mathcal{S}(E)$, $\lambda_{(n)}$ is given by \[\lambda_{(n)}(\xi_1 \otimes \dots \otimes \xi_n)=\xi_{1(0)} \otimes \dots \otimes \xi_{n(0)} \otimes \xi_{1(1)}\dots\xi_{n(1)}.\]
\end{corollary}

Let $A$ be a unital $\mathrm{C}^*$-algebra and $E$ be a right Hilbert $A$-module. The algebra $\mathcal{L}(E)$ is isomorphic to $E \otimes_A \mathrm{Hom}_A(E,A)$, yielding a unique linear map $\mathrm{Tr} : \mathcal{L}(E) \rightarrow A/[A,A]$ such that $\mathrm{Tr}(x \otimes f)=f(x)$ mod $[A,A]$. Let $\tau$ be any tracial linear functional on $A$. Then $\mathrm{Tr}_{\tau}=\tau \circ \mathrm{Tr}$ is a tracial linear functional on $\mathcal{L}(E)$.

Now let $\mathbb{R}\ni t \mapsto \sigma_{t}$ be a one-parameter automorphism group of $A$ and $\mathbb{R} \ni t \mapsto U_{t}$ be a one-parameter group of isometries on $E$ such that $U_{t}(\phi(a)\xi)=\phi(\sigma_{t}(a))U_{t}(\xi)$ and $\langle U_{t}(\xi),U_{t}(\eta)\rangle=\sigma_{t}\langle\xi,\eta\rangle$; we assume further that $\sigma$ and $U$ are strongly continuous. Then, by the universal property, there is, for each $t \in \mathbb{R}$, a unique automorphism $\delta_{t}:\clo_{E} \raro \clo_{E}$ such that $\delta_t(k_A(a))=k_A(\sigma_t(a))$ and $\delta_t(k_E(\xi))=k_E(U_t(\xi))$, $a \in A$, $\xi \in E$. The resulting one-parameter group $t \mapsto \delta_{t}$ is strongly continuous and is called the quasi-free dynamics on $\mathcal{O}_E$ associated to the module dynamics $U$.

\begin{remark}
We would only be interested when $\sigma$ is the trivial dynamics on $A$. In that case, a module dynamics $U$ is then a one-parameter group of isometries $t \mapsto U_t$ on $E$ such that $U_t(\phi(a)\xi)=\phi(a)U_t(\xi)$ and $\langle U_t(\xi),U_t(\eta) \rangle=\langle \xi,\eta \rangle$.
\end{remark}

With these in hand, we have the following theorem about the \textup{KMS} states on the Pimsner algebra.

\begin{theorem}
\label{thm:kmsstates}\cite{Neshveyev}*{Theorem 2.5} Let $\mathbb{R} \ni t \mapsto U_{t}$ be a one parameter group of isometries on $E$ satisfying the following conditions.
\begin{itemize}
	\item $U_t(\phi(a)\xi)=\phi(a)U_t(\xi)$ and $\langle U_t(\xi),U_t(\eta) \rangle=\langle \xi,\eta \rangle$, for $t \in \mathbb{R}$, $\xi,\eta \in E$;
	\item the vectors $\xi\in E$ such that $\mathrm{Sp}_{U}(\xi) \subset (0,\infty)$ form a dense subspace of $E$, where $\mathrm{Sp}_{U}(\xi)$ is the Arveson spectrum of $\xi$ with respect to $U$.
\end{itemize}
Let $\delta$ be the corresponding quasi-free dynamics on $\clo_{E}$ such that 
\begin{itemize}
	\item $\delta_{t}(k_E(\xi))=k_E(U_{t}(\xi))$ for $\xi \in E$ and
	\item $\delta_{t}(k_A(a))=k_A(a)$ for all $a \in A$,
\end{itemize} and suppose $\beta \in(0,\infty)$. 
\begin{itemize}
	\item If $\varphi$ is a $(\delta,\beta)$-\textup{KMS} state on $\clo_{E},$ then $\tau=\varphi \circ k_A$ is a tracial state on $A$, \[\mathrm{Tr}_{\tau}(\phi(a)e^{-\beta D}) \leq \tau(a), \text{ for } a \in A_+,\] and \[\mathrm{Tr}_{\tau}(\phi(a)e^{-\beta D})=\tau(a) \text{ for } a \in (\ker \phi)^{\perp}.\] Here $D$ is the generator of $U$ i.e., $U_{t}=e^{itD}$ and $\mathrm{Tr}_{\tau}(\phi(a)e^{-\infty D})=0$, by convention.
	\item Conversely, if $\tau$ is a tracial state on $A$ such that \[\mathrm{Tr}_{\tau}(\phi(a)e^{-\beta D}) \leq \tau(a), \text{ for } a \in A_+,\] and \[\mathrm{Tr}_{\tau}(\phi(a)e^{-\beta D})=\tau(a) \text{ for } a \in (\ker \phi)^{\perp},\] then there exists a unique $(\delta,\beta)$-\textup{KMS} state $\varphi$ on $\clo_{E}$ such that $\varphi \circ k_A=\tau$. Moreover, $\varphi$ is determined by $\tau$ through
	\begin{align}
	\begin{aligned}
	{}&\varphi(k_E(\xi_{1})\dots k_E(\xi_{m})k_E(\eta_{n})^*\dots k_E(\eta_{1})^*)\\
	={}&\tau(\langle \eta_{1} \ot \dots \ot \eta_{n},e^{-\beta D}\xi_{1} \ot \dots \ot e^{-\beta D}\xi_{n} \rangle), \text{ if } m=n,\\
	={}&0, \text{ otherwise. } 
\end{aligned}\label{eq:formula}
\end{align}
\end{itemize}
\end{theorem} 

\begin{remark}
We would be interested not in the existence of a \textup{KMS} state on the Pimsner algebra $\mathcal{O}_E$ but the fact that when it indeed does exist, it is of the form given above by Equation (\ref{eq:formula}) in Theorem \ref{thm:kmsstates}.
\end{remark} 

\begin{remark}
We remark that Arveson spectrum is defined, for example, in \cite{pedersen}*{8.1.6, page 385}.	
\end{remark}

\begin{theorem}
\label{thm:kmsequiv}
Let $G$ be a compact quantum group of Kac type, $(A,\alpha)$ be a $G$-$\mathrm{C}^*$-algebra, $(E,\phi,\lambda)$ be a $G$-equivariant $\mathrm{C}^*$-correspondence over the $G$-$\mathrm{C}^*$-algebra $(A,\alpha)$ and $\omega$ be the $G$-action on $\mathcal{O}_E$, as obtained in Theorem \textup{\ref{thm:maintheorem}}. Let $\delta$ be the quasi-free dynamics induced by the module dynamics $U$ satisfying the conditions as in Theorem \textup{\ref{thm:kmsstates}}. Let $U$ be $G$-equivariant, i.e., for all $t \in \mathbb{R}$, \[(U_t \otimes \mathrm{id}_{\mathrm{C}(G)})\circ \lambda=\lambda \circ U_t.\] Let $\varphi$ be a $(\delta,\beta)$-\textup{KMS} state on $\mathcal{O}_E$ and $\tau=\varphi \circ k_A$ be the tracial state on $A$ as in Theorem \textup{\ref{thm:kmsstates}}. Then $\varphi$ is $G$-equivariant if and only if $\tau$ is $G$-equivariant.
\end{theorem} 

\begin{proof}
We assume first that $\varphi$ is $G$-equivariant and show that $\tau=\varphi \circ k_A$ is $G$-equivariant too. To that end, we fix $a \in A$. Then we have 
\[
\begin{aligned}
	(\tau \otimes \mathrm{id}_{\mathrm{C}(G)})\alpha(a)={}&(\varphi \otimes \mathrm{id}_{\mathrm{C}(G)})(k_A \otimes \mathrm{id}_{\mathrm{C}(G)})\alpha(a)\\
	={}&(\varphi \otimes \mathrm{id}_{\mathrm{C}(G)})\omega(k_A(a))\\
	={}&\varphi(k_A(a)) 1_{\mathrm{C}(G)}\\
	={}&\tau(a) 1_{\mathrm{C}(G)},
\end{aligned}	
\]where the first equality uses the fact that $\tau=\varphi \circ k_A$; the second equality follows from Theorem \ref{thm:maintheorem}; the third equality is from our assumption that $\varphi$ is $G$-equivariant. 

We now assume that $\tau$ is $G$-equivariant and show that $\varphi$ too is $G$-equivariant. We will show that for any $x \in \mathcal{O}_E$,
\[
(\varphi \otimes h)\omega(x)=\varphi(x),
\] where $h$ is the Haar state of the CQG $G$. Before proving this, let us see how we can conclude the proof from this identity. We observe first that the identity can be written, using the standard convolution notation, as 
\[
\varphi \ast h=\varphi.
\]Now, for any $\psi \in \mathrm{C}(G)^*$,
\[
\varphi \ast \psi=(\varphi \ast h) \ast \psi=\varphi \ast (h \ast \psi)=(\varphi \ast h)=\varphi,
\]where the first equality uses the identity just above the computation; the second uses associativity of convolution; the third equality uses invariance of the Haar state $h$. Therefore, for all $\psi \in \mathrm{C}(G)^*$ and all $x \in \clo_{E}$,
\[
(\varphi \ast \psi)(x)=\varphi(x),
\]i.e.,
\[
\psi((\varphi \otimes \mathrm{id}_{\mathrm{C}(G)})\omega(x))=\psi(\varphi(x)1_{\mathrm{C}(G)}).
\]Therefore, we indeed have for all $x \in \clo_{E}$,
\[(\varphi \otimes \mathrm{id}_{\mathrm{C}(G)})\omega(x)=\varphi(x)1_{\mathrm{C}(G)},\]which is what we wanted. So we can now proceed to prove that for any $x \in \mathcal{O}_E$,
\begin{align}
\begin{aligned}
(\varphi \otimes h)\omega(x)=\varphi(x),\label{eq:toshow}
\end{aligned}
\end{align}holds. As the Hilbert $A$-module is assumed to be full and finitely generated, the linear span of the elements of the form $k_E(\xi_{1})\dots k_E(\xi_{m})k_E(\eta_{n})^*\dots k_E(\eta_{1})^* \in \clo_{E}$ is dense in $\clo_{E}$. Therefore, by linearity and continuity of the maps involved, it suffices to prove \eqref{eq:toshow} only for $k_E(\xi_{1})\dots k_E(\xi_{m})k_E(\eta_{n})^*\dots k_E(\eta_{1})^* \in \clo_{E}$. Moreover, $\xi_1,\dots,\xi_m,\eta_1,\dots,\eta_n$ further can be chosen from the spectral submodule $\mathcal{S}(E)$, where we have an algebraic coaction of $\mathbb{C}[G]$, hence allowing us to leverage Sweedler notation. However, we shall refrain from saying so and use Sweedler notation freely in what follows. We also remind the reader that $\varphi$ is determined through $\tau$ via the formula \eqref{eq:formula}, as in Theorem \ref{thm:kmsstates} above. We therefore begin by fixing $\xi_1,\dots,\xi_m,\eta_1,\dots,\eta_n \in E$ and assume first that $m \neq n$. Then we have
\[
\begin{aligned}
{}&(\varphi \otimes h)\omega(k_E(\xi_{1})\dots k_E(\xi_{m})k_E(\eta_{n})^*\dots k_E(\eta_{1})^*)\\
={}&(\varphi \otimes h)\Big(\omega(k_E(\xi_{1}))\dots \omega(k_E(\xi_{m}))\omega(k_E(\eta_{n}))^*\dots \omega(k_E(\eta_{1}))^\ast\Big)\\
={}&(\varphi \otimes h)\biggl((k_E \otimes \mathrm{id}_{\mathrm{C}(G)})\lambda(\xi_1)\dots(k_E \otimes \mathrm{id}_{\mathrm{C}(G)})\lambda(\xi_m)\\
&{}((k_E \otimes \mathrm{id}_{\mathrm{C}(G)})\lambda(\eta_n))^*\dots ((k_E \otimes \mathrm{id}_{\mathrm{C}(G)})\lambda(\eta_1))^*\biggr)\\
={}&\varphi\Big(k_E(\xi_{1(0)})\dots k_E(\xi_{m(0)})k_E(\eta_{n(0)})^*\dots k_E(\eta_{1(0)})^*\Big)h\Big(\xi_{1(1)}\dots\xi_{m(1)}\eta^*_{n(1)}\dots\eta^*_{1(1)}\Big)\\
={}&0\\
={}&\varphi(k_E(\xi_{1})\dots k_E(\xi_{m})k_E(\eta_{n})^*\dots k_E(\eta_{1})^*),
\end{aligned}	
\]where the second equality is by Theorem \ref{thm:maintheorem}; the fourth equality is because of $m\neq n$; the fifth equality is again because of $m\neq n$. So \eqref{eq:toshow} holds in this case. Now let us consider the case where $m=n$. We have 
\begin{align}
\begin{aligned}
{}&(\varphi \otimes h)\omega(k_E(\xi_{1})\dots k_E(\xi_{m})k_E(\eta_{m})^*\dots k_E(\eta_{1})^*)\\
={}&(\varphi \otimes h)\Big(\omega(k_E(\xi_{1}))\dots \omega(k_E(\xi_{m}))\omega(k_E(\eta_{m}))^*\dots \omega(k_E(\eta_{1}))^*\Big)\\
={}&(\varphi \otimes h)\biggl((k_E \otimes \mathrm{id}_{\mathrm{C}(G)})\lambda(\xi_1)\dots(k_E \otimes \mathrm{id}_{\mathrm{C}(G)})\lambda(\xi_m)\\
&{}((k_E \otimes \mathrm{id}_{\mathrm{C}(G)})\lambda(\eta_m))^*\dots ((k_E \otimes \mathrm{id}_{\mathrm{C}(G)})\lambda(\eta_1))^*\biggr)\\
={}&\varphi\Big(k_E(\xi_{1(0)})\dots k_E(\xi_{m(0)})k_E(\eta_{m(0)})^*\dots k_E(\eta_{1(0)})^*\Big)\\
&h\Big(\xi_{1(1)}\dots\xi_{m(1)}\eta^*_{m(1)}\dots\eta^*_{1(1)}\Big),
\end{aligned}\label{eq:a}
\end{align}where we have used nothing but Theorem \ref{thm:maintheorem}. Let us consider the expression
\allowdisplaybreaks{
\begin{align}
\begin{aligned}
{}&(\tau \otimes h)(\langle\lambda_m(\eta_{1}\otimes \dots \ot \eta_{m}),\lambda_m(e^{-\beta D}\xi_{1} \ot \dots e^{-\beta D}\xi_{m})\rangle)\\
={}&(\tau \ot h)\biggl(\langle\eta_{1(0)} \ot \dots \ot \eta_{m(0)},e^{-\beta D}\xi_{1(0)} \ot \dots \ot e^{-\beta D}\xi_{m(0)}\rangle  \otimes\\
&\eta_{m(1)}^*\dots\eta_{1(1)}^*\xi_{1(1)}\dots\xi_{m(1)}\biggr)\\
={}&\tau\biggl(\langle\eta_{1(0)} \ot \dots \ot \eta_{m(0)},e^{-\beta D}\xi_{1(0)} \ot \dots \ot e^{-\beta D}\xi_{m(0)}\rangle\biggr)\\
&h(\eta_{m(1)}^*\dots\eta_{1(1)}^*\xi_{1(1)}\dots\xi_{m(1)})\\
={}&\varphi(k_E(\xi_{1(0)}) \dots k_E(\xi_{m(0)})k_E(\eta_{m(0)})^*\dots k_E(\eta_{1(0)})^*)\\
&h(\xi_{1(1)}\dots\xi_{m(1)}\eta_{m(1)}^*\dots\eta_{1(1)}^*),
\end{aligned}\label{eq:b}
\end{align}
}where the first equality is by the expression of $\lambda_m$ as in Corollary \ref{cor:equivhigher} and the assumption that $U$ is $G$-equivariant; the third equality uses the traciality of $h$ and the expression of $\varphi$ in formula \eqref{eq:formula}. Now by combining \eqref{eq:a} and \eqref{eq:b}, we obtain
\begin{align}
\begin{aligned}
&{}(\varphi \ot h)\omega(k_E(\xi_{1}) \dots k_E(\xi_{m})k_E(\eta_{m})^* \dots k_E(\eta_{1})^*)\\
={}&(\tau \ot h)(\langle\lambda_m(\eta_{1} \ot \dots \ot \eta_{m}),\lambda_m(e^{-\beta D}\xi_{1} \ot \dots \otimes e^{-\beta D}\xi_{m})\rangle).	
\end{aligned}\label{eq:c}
\end{align}But the last expression of \eqref{eq:c} can be further simplified as follows.
\begin{align}
\begin{aligned}
{}&(\tau \ot h)(\langle\lambda_m(\eta_{1} \ot \dots \ot \eta_{m}),\lambda_m(e^{-\beta D}\xi_{1} \ot \dots \otimes e^{-\beta D}\xi_{m})\rangle)\\
={}&h((\tau \ot\mathrm{id}_{\mathrm{C}(G)})\alpha(\langle\eta_{1}\ot \dots \ot \eta_{m},e^{-\beta D}\xi_{1} \ot \dots \ot e^{-\beta D}\xi_{m}\rangle))\\
={}&h(\tau(\langle\eta_{1} \ot \dots \ot\eta_{m},e^{-\beta D}\xi_{1}\ot \dots \ot e^{-\beta D}\xi_{m}\rangle)1_{\mathrm{C}(G)})\\
={}&\tau(\langle\eta_{1} \ot \dots \ot\eta_{m},e^{-\beta D}\xi_{1}\ot \dots \ot e^{-\beta D}\xi_{m}\rangle)\\
={}&\varphi(k_E(\xi_{1}) \dots k_E(\xi_{m})k_E(\eta_{m})^*\dots k_E(\eta_{1})^*),
\end{aligned}\label{eq:d}
\end{align}where the first equality is by Corollary \ref{cor:equivhigher}; the second is by our assumption that $\tau$ is $G$-equivariant; the third is because $h$ is a state; and finally the fourth is by formula \eqref{eq:formula}. Now we combine \eqref{eq:c} and \eqref{eq:d}, and obtain
\[
\begin{aligned}
&{}(\varphi \ot h)\omega(k_E(\xi_{1}) \dots k_E(\xi_{m})k_E(\eta_{m})^* \dots k_E(\eta_{1})^*)\\
={}&\varphi(k_E(\xi_{1}) \dots k_E(\xi_{m})k_E(\eta_{m})^*\dots k_E(\eta_{1})^*),	
\end{aligned}
\]which shows that \eqref{eq:toshow} holds, settling this case as well for good. This completes the proof.
\end{proof}

In the remaining three sections, we provide applications of the results obtained in this and the previous sections to the notion of quantum symmetries of graphs.

\section{Applications to quantum symmetries of graphs I: Generalities} 
\label{sec:applicationsone}
In this section, we specialize to the example where the $\mathrm{C}^*$-correspondence comes from a finite graph, with some restrictions and see what the general results of the previous sections have to offer in this very concrete situation. So without further ado, we make contact via the following example.
	
\begin{example}{}
\label{exa:graph}\cite{muhly}*{Example 2.9}
Let $\clg=(\mathcal{G}^{1},\mathcal{G}^{0},r,s)$ be a finite, possibly with loops and multiple edges, directed graph. Here $\mathcal{G}^{1}, \mathcal{G}^{0}$ are the sets of the edges and vertices, respectively, and $r,s$ the range, source maps, respectively. Then $\mathcal{G}$ gives rise to a $\mathrm{C}^*$-correspondence over the $\mathrm{C}^*$-algebra $\mathrm{C}(\mathcal{G}^{0})$ as follows. $\mathrm{C}(\mathcal{G}^{1})$ is made into a right Hilbert $\mathrm{C}(\mathcal{G}^0)$-module via 
\[(\xi \cdot f)(e)=\xi(e)f(r(e)), \quad \langle \xi,\eta \rangle (v)=\sum_{r(e)=v}\overline{\xi(e)}\eta(e),\] where $\xi, \eta \in \mathrm{C}(\mathcal{G}^1)$, $f \in \mathrm{C}(\mathcal{G}^{0})$, $e \in \mathcal{G}^{1}$ and $v \in \mathcal{G}^0$. We define the unital $*$-homomorphism $\phi : \mathrm{C}(\mathcal{G}^0) \rightarrow \mathcal{L}(\mathrm{C}(\mathcal{G}^1))$ by \[\phi(f)(\xi)(e)=f(s(e))\xi(e),\] where $\xi \in \mathrm{C}(\mathcal{G}^1)$, $f \in \mathrm{C}(\mathcal{G}^0)$, $e \in \mathcal{G}^1$. Then since $\mathrm{C}(\mathcal{G}^1)$ is finite dimensional, the pair $(\mathrm{C}(\mathcal{G}^1),\phi)$ forms a $\mathrm{C}^*$-correspondence over $\mathrm{C}(\mathcal{G}^0)$. The corresponding Pimsner algebra $\mathcal{O}_{\mathrm{C}(\mathcal{G}^1)}$ is isomorphic to the graph $\mathrm{C}^*$-algebra  denoted by $\mathrm{C}^*(\clg)$. The gauge action on the Pimsner algebra $\mathrm{C}^*(\clg)$ coincides with the usual gauge action on the graph $\mathrm{C}^*$-algebra $\mathrm{C}^*(\clg)$. 
\end{example}
	
Now we shall recall a few facts about the \textup{KMS} states of graph $\mathrm{C}^*$-algebras coming from finite graphs, possibly with multiple edges but without sources;  by the term without sources, we shall mean that the map $r$ is onto. To that end, let $\clg=(\mathcal{G}^{0},\mathcal{G}^{1},r,s)$ be a finite, directed graph, without sources and $\mathrm{C}^*(\clg)$ be the corresponding graph $\mathrm{C}^*$-algebra; we also denote the adjacency matrix of $\clg$ by $\cld$ and the spectral radius of $\cld$ by $\rho(\cld)$. Then we have the following proposition.

\begin{proposition}\cite{Joardar2}*{Proposition 2.4}
The graph $\mathrm{C}^*$-algebra $\mathrm{C}^*(\clg)$ has a \textup{KMS}-${\ln(\rho(\cld))}$ state if and only if $\rho(\cld)$ is an eigenvalue of $\cld$ with eigenvectors having all its entries non-negative. 
\end{proposition}

We shall call a \textup{KMS}-$\ln(\rho(\cld))$ state $\varphi$ \textit{distinguished} if $\varphi(p_{v_{i}})=\frac{1}{n}$ for all $i=1,\dots,n$, where $v_{1},\dots,v_{n}$ are the vertices of the underlying graph $\mathcal{G}$. The distinguished \textup{KMS} states are in abundance. For example, any regular graph admits such a distinguished \textup{KMS} state on its $\mathrm{C}^{\ast}$-algebra. Indeed, for a regular graph the spectral radius $\rho(\cld)$ is an eigenvalue with eigenspace spanned by the vector $(1,\dots,1)$; see \cite{Joardar1}*{Theorem 3.6} for more details. 

\begin{remarks}{}\hfill
\label{rem:kac} 
\begin{itemize}
	\item Let us assume that $\rho(\cld)$ is an eigenvalue of $\cld$ so that $\mathrm{C}^*(\clg)$ admits a \textup{KMS}-$\ln(\rho(\cld))$ state. If $\cld$ is an $n \times n$ matrix and $(\mu_{1},\dots,\mu_{n})$ is a normalized eigenvector of $\cld$ with eigenvalue $\rho(\cld)$ (i.e., $\sum_{i=1}^{n}\mu_{i}=1$) then the corresponding \textup{KMS}-$\ln(\rho(\cld))$ state $\varphi$ satisfies 
	\[
	\varphi(p_{v_{i}})=\mu_{i} \quad i=1,\dots,n,
	\]
	where $v_{1},\dots,v_{n}$ are the vertices of $\clg$. In fact, with the notation of Theorem \ref{thm:kmsstates}, the tracial state $\tau$ on $\mathrm{C}(\mathcal{G}^{0})$ corresponding to the \textup{KMS} state $\varphi$ is given by
	\[
	\tau(\delta_{v_{i}})=\mu_{i} \quad i=1,\dots,n.
	\]
	\item The module dynamics in this case is nothing but the scalar dynamics, i.e., $U_t=e^{it}$, $t \in \mathbb{R}$, for the gauge action and therefore it is $G$-equivariant for any compact quantum group $G$ acting on the correspondence.
	\item If for a graph $\clg$, $\mathrm{C}^{\ast}(\clg)$ has a distinguished \textup{KMS} state $\varphi$, then the corresponding tracial state $\tau$ on $\mathrm{C}(\mathcal{G}^{0})$ is given by 
	\[
	\tau(\delta_{v_{i}})=\frac{1}{n} \quad i=1,\dots,n.
	\]So for any $G$-equivariant $\mathrm{C}^{\ast}$-correspondence arising from a finite, directed graph $\clg$, possibly with multiple edges but without any source, such that $\mathrm{C}^{\ast}(\clg)$ admits a distinguished \textup{KMS} state $\varphi$, the above tracial state $\tau$ on $\mathrm{C}(\mathcal{G}^{0})$ is always $G$-equivariant.
\end{itemize}
\end{remarks} 

For the next proposition again, $\cld$ is the adjacency matrix of the graph $\clg$ and $\rho(\cld)$ is its spectral radius.

\begin{proposition}{}
\label{prop:kac}
Let $\clg=(\mathcal{G}^{0},\mathcal{G}^{1},r,s)$ be a finite graph, possibly with multiple edges, but without any source or sink, so that the corresponding graph $\mathrm{C}^*$-algebra $\mathrm{C}^*(\clg)$ has a distinguished \textup{KMS}-$\ln(\rho(\cld))$ state $\varphi$. Let $G$ be a compact quantum group such that 
\begin{itemize}
	\item the $\mathrm{C}^*$-correspondence $(\mathrm{C}(\mathcal{G}^1),\phi)$, arising from $\clg$ is $G$-equivariant;
	\item the $G$-action $\omega$ obtained from Theorem \textup{\ref{thm:maintheorem}} on $\mathrm{C}^*(\mathcal{G})$ is faithful.
\end{itemize} Then $G$ is of Kac type if and only if the \textup{KMS} state $\varphi$ is $G$-equivariant.
\end{proposition}

\begin{proof}
Let $G$ be compact quantum group of Kac type. Then by (2) and (3) of Remark \ref{rem:kac} above, all the conditions of Theorem \ref{thm:kmsequiv} are satisfied and the \textup{KMS} state $\varphi$ on $\mathrm{C}^*(\clg)$ is $G$-equivariant.

Conversely, we assume that the \textup{KMS} state $\varphi$ is $G$-equivariant. Then it follows from the proof of \cite{idaqp}*{Proposition 3.8} that $G$ is a quantum subgroup of $\textup{U}_{n}^{+}$ (observe that the matrix $F^{\clg}$ in \cite{idaqp}*{Proposition 3.8} is the identity matrix). Hence $G$ is a compact quantum group of Kac type.
\end{proof} 

\begin{corollary}
\label{cor:utranspose}
Let $G$ be a compact matrix pseudogroup such that the fundamental corepresentation $u$ is unitary. Then $G$ is of Kac type if and only if $u^{t}$ is also unitary.
\end{corollary}  

\begin{proof}
We consider the $\mathrm{C}^*$-correspondence coming from the graph of the Cuntz algebra with $n$-generators. Then since the graph has only one vertex, $\mathrm{C}(\mathcal{G}^0)=\mathbb{C}$ and hence the Hilbert $\mathrm{C}^*$-module $\mathrm{C}(\mathcal{G}^{1})$ is an $n$-dimensional Hilbert space. So if $(u_{ij})_{i,j=1,\dots,n}$ is the unitary matrix corresponding to the fundamental unitary corepresentation $u$, denoting an orthonormal basis of $\mathrm{C}(\mathcal{G}^{1})$ by $e_{1},\dots,e_{n}$, we get a $G$-action on $\mathrm{C}(\mathcal{G}^1)$ by
\[
\lambda(e_{j})=\sum_{i=1}^{n}e_{i}\ot u_{ij}, \quad j=1,\dots,n.
\]
The $G$-action $\alpha$ on $\mathrm{C}(\mathcal{G}^{0})=\mathbb{C}$ is the trivial one and it is easy to see that the above $\alpha$, and $\lambda$ make the $\mathrm{C}^*$-correspondence $(\mathrm{C}(\mathcal{G}^1),\phi)$ a $G$-equivariant $\mathrm{C}^*$-correspondence. Thus there is a $G$-action $\omega$ on $\clo_{n}$ given on the generators by 
\[
\omega(S_{j})=\sum_{i=1}^{n}S_{i} \ot u_{ij}, \quad j=1,\dots,n.
\]
But the unique \textup{KMS} state on $\clo_{n}$ is $G$-equivariant if and only if $u^{t}=((u_{ji}))$ is unitary (see \cite{Joardar1}). Now the unique \textup{KMS} state is a distinguished \textup{KMS} state in the above sense. So applying Proposition \ref{prop:kac}, we obtain the desired conclusion.
\end{proof} 

\begin{remark}
The conclusion of the Corollary \ref{cor:utranspose} is well-known as can be easily seen by applying $\kappa$ to $u^*u=uu^*$ and using the fact that for a compact quantum group of Kac type, $\kappa$ is involutive. But our proof is from the point of view of quantum symmetries. 
\end{remark} 

\section{Applications to quantum symmetries of graphs II: The case of simple graphs}
\label{sec:applicationstwo}
Let $A$ be a unital $\mathrm{C}^*$-algebra and $(E,\phi)$ be a $\mathrm{C}^*$-correspondence over $A$. We consider the category of all compact quantum groups $G$ such that $(E,\phi)$ is a $G$-equivariant $\mathrm{C}^*$-correspondence. Although this category is an interesting one, it has a drawback that in general it might fail to admit a universal object. For example, if one considers the $\mathrm{C}^*$-correspondence coming from the graph of the Cuntz algebra with $n$-generators, any Wang algebra $A_{u}(Q)$ can be made into an object of this category and therefore it does not admit a universal object. In this section, we shall consider the $\mathrm{C}^*$-correspondence coming from a finite graph, possibly with multiple edges but without any source, and make the category more restrictive so that the modified category admits a  universal object. Therefore let us begin with the following definition.

\begin{definition}
\label{def:qsymban}
Let $\mathcal{G}$ be a finite, directed graph, with $n$ edges and $m$ vertices (without loops or multiple edges). The compact quantum group $\mathrm{Aut}^{+}_{\textup{Ban}}(\clg)$ is defined to be the quotient $\textup{S}^+_m/(UD-DU)$, where $U=(q_{vw})_{v,w \in \mathcal{G}^{0}}$, and $D$ is the adjacency matrix for $\clg$. The coproduct on the generators is given by $\Delta_{\mathrm{Aut}^{+}_{\textup{Ban}}(\clg)}(q_{vw})=\sum_{u\in \mathcal{G}^{0}}q_{vu} \ot q_{uw}$.
\end{definition}

The relation $UD=DU$, when expanded out, yields the following explicit
description of the $\mathrm{C}^{*}$-algebra
$\mathrm{C}(\mathrm{Aut}^{+}_{\textup{Ban}}(\clg))$. 
	
\begin{lemma} 
\label{thm:genrel}\cite{Fulton}*{Lemma 3.1.1} 
The underlying $\mathrm{C}^*$-algebra $\mathrm{C}(\mathrm{Aut}^{+}_{\textup{Ban}}(\clg))$ of the quantum group $\mathrm{Aut}^{+}_{\textup{Ban}}(\clg)$ for a finite, directed graph $\clg$ with $n$ edges and $m$ vertices (without loops or multiple edges) is the universal $\mathrm{C}^*$-algebra generated by $(q_{vw})_{v,w\in \mathcal{G}^{0}}$ satisfying the following relations.
\begin{align}
q_{vw}^*=q_{vw}, \ q_{vw}q_{vu}=\delta_{wu}q_{vw}, \ q_{vw}q_{uw}=\delta_{vu}q_{vw}, \quad u,v,w \in \mathcal{G}^0,\label{eq:r1}
\end{align} 
\begin{align}
\sum_{w \in \mathcal{G}^{0}}q_{vw}=\sum_{w \in \mathcal{G}^{0}}q_{wv}=1, \quad v \in \mathcal{G}^{0},\label{eq:r2}
\end{align}
\begin{align}
\begin{aligned}
&{}q_{s(e)v}q_{r(e)w}=q_{r(e)w}q_{s(e)v}=0, \quad e \in \mathcal{G}^{1}, (v,w) \not \in \mathcal{G}^{1},\\
&{}q_{vs(e)}q_{wr(e)}=q_{wr(e)}q_{vs(e)}=0, \quad e \in \mathcal{G}^{1}, (v,w) \not \in \mathcal{G}^{1}.
\end{aligned}\label{eq:r3}
\end{align}
The comultiplication on the generators is given by 
\begin{align}
\Delta_{\mathrm{Aut}^{+}_{\textup{Ban}}(\clg)}(q_{vw})=\sum_{u\in \mathcal{G}^{0}}q_{vu} \ot q_{uw}.\label{eq:del1}
\end{align} The action on the graph is given by \[\alpha(p_{v})=\sum_{w\in \mathcal{G}^{0}}p_{w} \ot q_{wv}, \ v \in \mathcal{G}^{0}.\]
\end{lemma}
	
\begin{remark}
\label{rem:transpose}
Since $\mathrm{Aut}^{+}_{\textup{Ban}}(\clg)$ is a quantum subgroup of $\textup{S}_{m}^{+}$, it is of Kac type.
\end{remark}
	
\begin{remark}
\label{rem:bichon}
There is another notion of quantum symmetry for directed, simple graphs due to
Bichon. We refer the reader to \cite{Bichon} for details. The quantum
automorphism group of a simple, directed graph $\clg$ in the sense of Bichon, to
be denoted by $\mathrm{Aut}^{+}_{\textup{Bic}}(\clg)$, is a quantum subgroup of $\mathrm{Aut}^{+}_{\textup{Ban}}(\clg)$. The underlying $\mathrm{C}^*$-algebra $\mathrm{C}(\mathrm{Aut}^{+}_{\textup{Bic}}(\clg))$ is again generated by $q_{vw}$, $v,w\in \mathcal{G}^{0}$ and satisfy the relations in Theorem \ref{thm:genrel}, as well as the following additional relations.
\begin{align}
	q_{s(e)s(f)}q_{r(e)r(f)}=q_{r(e)r(f)}q_{s(e)s(f)}, \quad e,f\in \mathcal{G}^{1}.\label{eq:r4}
\end{align}
The comultiplication on the generators is again given by 
\begin{align}
	\Delta_{\mathrm{Aut}^{+}_{\textup{Bic}}(\clg)}(q_{vw})=\sum_{u\in \mathcal{G}^0}q_{vu}\ot q_{uw}.\label{eq:del2}
\end{align}	
\end{remark}

\begin{definition}
\label{def:Bancat} 
Let $\mathcal{G}=(\mathcal{G}^1,\mathcal{G}^0,r,s)$ be a finite, directed graph, without loops or multiple edges. We define the category $\clc_{\textup{Ban}}(\clg)$ as follows. 
\begin{itemize}
	\item An object of $\clc_{\textup{Ban}}(\clg)$ is a triple $(G,\alpha,\lambda)$, where $G$ is a compact quantum group, $\alpha : \mathrm{C}(\mathcal{G}^0) \rightarrow \mathrm{C}(\mathcal{G}^0) \otimes \mathrm{C}(G)$ is a unital $*$-homomorphism and $\lambda : \mathrm{C}(\mathcal{G}^1) \rightarrow \mathrm{C}(\mathcal{G}^1) \otimes \mathrm{C}(G)$ is a linear map satisfying the following conditions.
	\begin{itemize}
		\item The pair $(\mathrm{C}(\mathcal{G}^0),\alpha)$ is a $G$-$\mathrm{C}^*$-algebra and the action $\alpha$ is faithful;
		\item the triple $(\mathrm{C}(\mathcal{G}^1),\phi,\lambda)$ is a $G$-equivariant $\mathrm{C}^*$-correspondence over the $G$-$\mathrm{C}^*$-algebra $(\mathrm{C}(\mathcal{G}^0),\alpha)$;
		\item for all $f \in \mathrm{C}(\mathcal{G}^{0})$, $(r_{\ast}\ot \mathrm{id}_{\mathrm{C}(G)})\alpha(f)=\lambda(r_{\ast}(f))$, where $r_{\ast} :\mathrm{C}(\mathcal{G}^{0}) \raro \mathrm{C}(\mathcal{G}^{1})$ is the $*$-homomorphism given by $r_{\ast}(f)(e)=f(r(e))$ for $e\in \mathcal{G}^{1}$.
	\end{itemize}
	\item Let $(G_1,\alpha_1,\lambda_1)$ and $(G_2,\alpha_2,\lambda_2)$ be two objects of the category $\clc_{\textup{Ban}}(\clg)$. A morphism $f : (G_1,\alpha_1,\lambda_1) \rightarrow (G_2,\alpha_2,\lambda_2)$ in $\clc_{\textup{Ban}}(\clg)$ is by definition a Hopf $*$-homomorphism $f : \mathrm{C}(G_2) \rightarrow \mathrm{C}(G_1)$ such that 
	\begin{itemize}
		\item $(\mathrm{id}_{\mathrm{C}(\mathcal{G}^0)} \otimes f)\alpha_2=\alpha_1$;
		\item $(\mathrm{id}_{\mathrm{C}(\mathcal{G}^1)} \otimes f)\lambda_2=\lambda_1$.
	\end{itemize}
\end{itemize}
\end{definition}

\begin{definition}
\label{def:Bichcat} 
Let $\mathcal{G}=(\mathcal{G}^1,\mathcal{G}^0,r,s)$ be a finite, directed graph, without loops or multiple edges. We define the category $\clc_{\textup{Bic}}(\clg)$ as follows. 
\begin{itemize}
	\item An object of $\clc_{\textup{Bic}}(\clg)$ is a triple $(G,\alpha,\lambda)$, where $G$ is a compact quantum group, $\alpha : \mathrm{C}(\mathcal{G}^0) \rightarrow \mathrm{C}(\mathcal{G}^0) \otimes \mathrm{C}(G)$ is a unital $*$-homomorphism and $\lambda : \mathrm{C}(\mathcal{G}^1) \rightarrow \mathrm{C}(\mathcal{G}^1) \otimes \mathrm{C}(G)$ is a linear map satisfying the following conditions.
	\begin{itemize}
		\item The pair $(\mathrm{C}(\mathcal{G}^0),\alpha)$ is a $G$-$\mathrm{C}^*$-algebra and the action $\alpha$ is faithful;
		\item the triple $(\mathrm{C}(\mathcal{G}^1),\phi,\lambda)$ is a $G$-equivariant $\mathrm{C}^*$-correspondence over the $G$-$\mathrm{C}^*$-algebra $(\mathrm{C}(\mathcal{G}^0),\alpha)$;
		\item the pair $(\mathrm{C}(\mathcal{G}^1),\lambda)$ is a $G$-$\mathrm{C}^*$-algebra.
	\end{itemize}
	\item Let $(G_1,\alpha_1,\lambda_1)$ and $(G_2,\alpha_2,\lambda_2)$ be two objects of the category $\clc_{\textup{Bic}}(\clg)$. A morphism $f : (G_1,\alpha_1,\lambda_1) \rightarrow (G_2,\alpha_2,\lambda_2)$ in $\clc_{\textup{Bic}}(\clg)$ is again by definition a Hopf $*$-homomorphism $f : \mathrm{C}(G_2) \rightarrow \mathrm{C}(G_1)$ such that 
	\begin{itemize}
		\item $(\mathrm{id}_{\mathrm{C}(\mathcal{G}^0)} \otimes f)\alpha_2=\alpha_1$;
		\item $(\mathrm{id}_{\mathrm{C}(\mathcal{G}^1)} \otimes f)\lambda_2=\lambda_1$.
	\end{itemize}
\end{itemize}
\end{definition}


The notations used for the above two categories are justified by the following theorems.

\begin{theorem}
\label{thm:Bancorrespondence}
A universal object $(G_{\textup{Ban}},\alpha_{\textup{Ban}},\lambda_{\textup{Ban}})$ in the category $\clc_{\textup{Ban}}(\clg)$ exists. Moreover, $G_{\textup{Ban}}$ is isomorphic to $\mathrm{Aut}^{+}_{\textup{Ban}}(\clg)$, the quantum automorphism group of $\clg$ in the sense of Banica.
\end{theorem}  

\begin{theorem}
\label{thm:Biccorrespondence}
A universal object $(G_{\textup{Bic}},\alpha_{\textup{Bic}},\lambda_{\textup{Bic}})$ in the category $\clc_{\textup{Bic}}(\clg)$ exists. Moreover, $G_{\textup{Bic}}$ is isomorphic to $\mathrm{Aut}^{+}_{\textup{Bic}}(\clg)$, the quantum automorphism group of $\clg$ in the sense of Bichon.
\end{theorem}

\begin{proof}[Proof of Theorem \ref{thm:Bancorrespondence}]
We have to show first that $\mathrm{Aut}^{+}_{\textup{Ban}}(\clg)$ can be made into an object of the category $\clc_{\textup{Ban}}(\clg)$. We proceed to do so now. As in Theorem \ref{thm:genrel}, we denote the generators of the underlying $\mathrm{C}^*$-algebra $\mathrm{C}(\mathrm{Aut}^{+}_{\textup{Ban}}(\clg))$ of $\mathrm{Aut}^{+}_{\textup{Ban}}(\clg)$ by $\{q_{wv}\}_{v,w \in \mathcal{G}^0}$. Now we define $\lambda_{\textup{Ban}}$ and $\alpha_{\textup{Ban}}$ as
\begin{align}
\lambda_{\textup{Ban}}(\delta_{e})=\sum_{f\in \mathcal{G}^1}\delta_{f}\ot q_{s(f)s(e)}q_{r(f)r(e)}, \quad e \in \mathcal{G}^1,\label{eq:lambda}
\end{align}
\begin{align}
\alpha_{\textup{Ban}}(\delta_{v})=\sum_{w\in \mathcal{G}^0}\delta_{w}\ot q_{wv}, \quad v \in \mathcal{G}^0.\label{eq:alpha}
\end{align}The relations \eqref{eq:r1},\eqref{eq:r2} and \eqref{eq:r3} imply that $\alpha_{\textup{Ban}}$ is a unital $*$-homomorphism. The formula \eqref{eq:del1} for the comultiplication shows coassociativity of both $\alpha_{\textup{Ban}}$ and $\lambda_{\textup{Ban}}$. The Podle\'s conditions for $\alpha_{\textup{Ban}}$ and $\lambda_{\textup{Ban}}$ follow from the facts that $\mathrm{C}(\mathcal{G}^0)$ and $\mathrm{C}(\mathcal{G}^1)$ are finite dimensional, respectively. Furthermore, the fact that $\{q_{vw}\}_{v,w\in \mathcal{G}^0}$ are generators of $\mathrm{C}(\mathrm{Aut}^{+}_{\textup{Ban}}(\clg))$ implies that $\alpha_{\textup{Ban}}$ is a faithful $G$-action on $\mathrm{C}(\mathcal{G}^0)$. Therefore, to show that $(\mathrm{Aut}^{+}_{\textup{Ban}}(\clg),\alpha_{\textup{Ban}},\lambda_{\textup{Ban}})$ is an object of the category $\clc_{\textup{Ban}}(\clg)$, it suffices to show
\begin{align}
\lambda_{\textup{Ban}}(\phi(\delta_{v})\delta_{e})=(\phi \otimes \mathrm{id}_{\mathrm{C}(G)})\alpha_{\textup{Ban}}(\delta_{v})\lambda_{\textup{Ban}}(\delta_{e}), \quad v \in \mathcal{G}^0, e \in \mathcal{G}^1,\label{eq:e}
\end{align}
\begin{align}
\lambda_{\textup{Ban}}(\delta_{e}\delta_{v})=\lambda_{\textup{Ban}}(\delta_{e})\alpha(\delta_{v}), \quad v \in \mathcal{G}^0, e \in \mathcal{G}^1,\label{eq:f}
\end{align} 
\begin{align}
\langle\lambda_{\textup{Ban}}(\delta_{e}),\lambda_{\textup{Ban}}(\delta_{f})\rangle=\alpha_{\textup{Ban}}(\langle\delta_{e},\delta_{f}\rangle), \quad e, f \in \mathcal{G}^1 \label{eq:g}
\end{align}
\begin{align}
(r_{\ast} \otimes \mathrm{id})\alpha_{\textup{Ban}}(f)=\lambda_{\textup{Ban}}(r_{\ast}(f)), \quad f \in \mathrm{C}(\mathcal{G}^0). \label{eq:h}	
\end{align}We remind the reader that $\phi$ is given as in Example \ref{exa:graph}. Now to check identity \eqref{eq:e}, we note that $\phi(\delta_{v})\delta_{e}=\delta_{s(e),v}\delta_{e}$; hence the left-hand side of \eqref{eq:e} reduces to
\begin{align}
\lambda_{\textup{Ban}}(\phi(\delta_{v})\delta_{e})=\delta_{s(e),v}\sum_{f\in \mathcal{G}^1}\delta_{f}\ot q_{s(f)s(e)}q_{r(f)r(e)}.\label{eq:i}
\end{align}For the right-hand side of \eqref{eq:e}, we have the following expression.
\begin{align}
\begin{aligned}
&{}(\phi \otimes \mathrm{id}_{\mathrm{C}(G)})\alpha_{\textup{Ban}}(\delta_{v})\lambda_{\textup{Ban}}(\delta_{e})\\
={}&(\sum_{w\in \mathcal{G}^0}\phi(\delta_{w}) \ot q_{wv})(\sum_{f\in \mathcal{G}^1}\delta_{f}\ot q_{s(f)s(e)}q_{r(f)r(e)})\\
={}&\sum_{w\in \mathcal{G}^0}\sum_{w=s(f)}\delta_{f}\ot q_{wv}q_{s(f)s(e)}q_{r(f)r(e)}.
\end{aligned}\label{eq:j}
\end{align}But for $v \neq s(e)$, $q_{wv}q_{s(f)s(e)}=0$ for all $w=s(f)$, from the relation \eqref{eq:r1}; in that case, the final expression of \eqref{eq:j} is zero, coinciding with \eqref{eq:i} and so we have \eqref{eq:e}. For $v=s(e)$, the final expression of \eqref{eq:j} reduces to
\begin{align}
\sum_{w\in \mathcal{G}^0}\sum_{w=s(f)}\delta_{f}\ot q_{ws(e)}q_{s(f)s(e)}q_{r(f)r(e)}\label{eq:k}
\end{align}On the other hand, for $v=s(e)$, we now consider \eqref{eq:i}. 
\begin{align}
\begin{aligned}
\lambda_{\textup{Ban}}(\phi(\delta_{v})\delta_{e})={}&\delta_{s(e),v}\sum_{f\in \mathcal{G}^1}\delta_{f}\ot q_{s(f)s(e)}q_{r(f)r(e)}\\
={}&\sum_{f \in \mathcal{G}^1}\delta_{f}\ot q_{s(f)s(e)}q_{r(f)r(e)}\\
={}&\sum_{f\in \mathcal{G}^1}\delta_{f}\ot \sum_{w \in \mathcal{G}^0}q_{ws(e)}q_{s(f)s(e)}q_{r(f)r(e)}\\
={}&\sum_{w\in \mathcal{G}^0}\sum_{w=s(f)}\delta_{f}\ot q_{ws(e)}q_{s(f)s(e)}q_{r(f)r(e)},
\end{aligned}\label{eq:l}
\end{align}where the third equality is by \eqref{eq:r2}; and the fourth equality is by \eqref{eq:r1}. However, the final expression in \eqref{eq:l} is nothing but \eqref{eq:k}. Therefore we indeed have \eqref{eq:e}. We leave the checking of \eqref{eq:f} to the reader, which is similar to what has just been done for \eqref{eq:e}. To check the identity \eqref{eq:g}, we note that \[\langle\delta_{e},\delta_{f}\rangle=\delta_{e,f}\delta_{r(e)}.\] Then it reduces to the calculations already done in \cite{Web}*{Subsection 4.1.2}. So we are left with the checking of \eqref{eq:h}. To that end, we note that 
\[
r_{\ast}(\delta_{v})=\sum_{e \in r^{-1}(v)}\delta_{e}.
\]Therefore, the right-hand side of \eqref{eq:h} becomes
\begin{align}
\begin{aligned}
\lambda_{\textup{Ban}}(r_{\ast}(\delta_{v}))={}&\sum_{e \in r^{-1}(v)}\sum_{f \in \mathcal{G}^1}\delta_{f}\ot q_{s(f)s(e)}q_{r(f)r(e)}\\
={}& \sum_{f \in \mathcal{G}^1}\delta_{f}\ot \Big(\sum_{e \in r^{-1}(v)}q_{s(f)s(e)}q_{r(f)r(e)}\Big)
\end{aligned}\label{eq:m}
\end{align}On the other hand, the left-hand side of \eqref{eq:h} reduces to 
\begin{align}
\begin{aligned}
(r_{\ast}\otimes \mathrm{id}_{\mathrm{C}(G)})\alpha_{\textup{Ban}}(\delta_{v})={}& (r_{\ast}\ot\mathrm{id})(\sum_{w\in \mathcal{G}^0}\delta_{w}\ot q_{wv})\\
={}&\sum_{w\in \mathcal{G}^0}\Big(\sum_{f \in r^{-1}(w)}\delta_{f}\ot q_{wv}\Big)
\end{aligned}\label{eq:n}
\end{align}Now we observe that the set $\{f \in \mathcal{G}^1 \mid r(f)=w, w \in \mathcal{G}^0\}$ is the whole of $\mathcal{G}^1$ and therefore the last expression in \eqref{eq:n} reduces to
\begin{align}
\sum_{f \in \mathcal{G}^1}\delta_{f}\ot q_{r(f)v}.\label{eq:o}
\end{align}We also have, for each $f \in \mathcal{G}^1$, 
\begin{align}
q_{r(f)v}=\sum_{w \in \mathcal{G}^0}q_{s(f)w}q_{r(f)v};\label{eq:p}
\end{align}moreover, since $q_{s(f)w}q_{r(f)v}=0$ whenever $(v,w)$ is not an edge, 
\begin{align}
\sum_{w \in \mathcal{G}^0}q_{s(f)w}q_{r(f)v}=\sum_{e \in r^{-1}(v)}q_{s(f)s(e)}q_{r(f)r(e)}.\label{eq:q}
\end{align}Combining \eqref{eq:n}, \eqref{eq:o}, \eqref{eq:p} and \eqref{eq:q}, we obtain
\[
(r_{\ast}\otimes \mathrm{id}_{\mathrm{C}(G)})\alpha_{\textup{Ban}}(\delta_{v})=\sum_{f \in \mathcal{G}^1}\delta_{f}\otimes \Big(\sum_{e \in r^{-1}(v)}q_{s(f)s(e)}q_{r(f)r(e)}\Big),
\]which is exactly the last expression in \eqref{eq:m}. Therefore \eqref{eq:h} holds. This proves that indeed $(\mathrm{Aut}^{+}_{\textup{Ban}}(\clg),\alpha_{\textup{Ban}},\lambda_{\textup{Ban}})$ is an object in the category  $\clc_{\textup{Ban}}(\clg)$. 

Now we turn to the proof of the fact that $(\mathrm{Aut}^{+}_{\textup{Ban}}(\clg),\alpha_{\textup{Ban}},\lambda_{\textup{Ban}})$ is indeed the universal object in the category $\clc_{\textup{Ban}}(\clg)$. To that end, let $(G,\alpha,\lambda)$ be an object of the category $\clc_{\textup{Ban}}(\clg)$.
\[
\lambda(\delta_{e})=\sum_{f \in \mathcal{G}^1}\delta_{f}\ot a_{fe}, \quad a_{fe} \in \mathrm{C}(G), \quad e\in \mathcal{G}^1
\]
\[
\alpha(\delta_{v})=\sum_{w\in \mathcal{G}^0}\delta_{w}\ot b_{wv}, \quad b_{vw} \in \mathrm{C}(G), \quad v \in \mathcal{G}^0.
\]Now, as already observed, for any $v \in \mathcal{G}^0$,
\[ r_{\ast}(\delta_{v})=\sum_{e \in r^{-1}(v)}\delta_{e}.\]As the graph does not have any multiple edges or sources, for $w \in \mathcal{G}^0$ with $s(e)=w, r(e)=v$, \[\phi(\delta_{w})r_{\ast}(\delta_{v})=\delta_{e}.\]Hence we have
\begin{align}
\lambda(\phi(\delta_{w})r_{\ast}(\delta_{v}))=\sum_{f\in \mathcal{G}^1}\delta_{f}\ot a_{fe}, \text{ provided } s(e)=w,r(e)=v.
\end{align}On the other hand, since $(r_{\ast}\otimes \mathrm{id}_{\mathrm{C}(G)})\circ\alpha=\lambda\circ r_{\ast}$ and $\lambda(\phi(f)\xi)=(\phi \otimes \mathrm{id}_{\mathrm{C}(G)})\alpha(f)\lambda(\xi)$ for $f \in \mathrm{C}(\mathcal{G}^0)$ and $\xi \in \mathrm{C}(\mathcal{G}^1)$,
\begin{align} 
\begin{aligned}
\lambda(\phi(\delta_{w})r_{\ast}(\delta_{v}))={}&(\phi \otimes \mathrm{id}_{\mathrm{C}(G)})\alpha(\delta_{w})(r_{\ast}\ot\mathrm{id}_{\mathrm{C}(G)})\alpha(\delta_{v})\\
={}&\sum_{f : u \rightarrow p}\delta_{f}\ot b_{uw}b_{pv}. 
\end{aligned}
\end{align}By comparing coefficients, we get
\begin{align}
\lambda(\delta_{e})=\sum_{f \in \mathcal{G}^1}\delta_{f}\ot b_{s(f)s(e)}b_{r(f)r(e)}.
\end{align}Then the fact that $(b_{vw})_{v,w \in \mathcal{G}^0}$ satisfy the relations of $\mathrm{C}(\mathrm{Aut}^{+}_{\textup{Ban}}(\clg))$ can be checked, along the lines of \cite{Web}*{Subsection 4.3.3}, using \[\langle\lambda(\delta_{e}),\lambda(\delta_{e})\rangle=\alpha(\langle\delta_{e},\delta_{e}\rangle),\] together with the fact that\[\langle\delta_{e},\delta_{e}\rangle=\delta_{r(e)};\]we leave this to the reader. We therefore have shown the universality of the triple $(\mathrm{Aut}^{+}_{\textup{Ban}}(\clg),\alpha_{\textup{Ban}},\lambda_{\textup{Ban}})$, thus completing the proof of the theorem.
\end{proof}

\begin{remark}
The compact quantum group $\mathrm{Aut}^{+}_{\textup{Ban}}(\clg)$ is of Kac type; moreover, by Theorem \ref{thm:Bancorrespondence}$,  (\mathrm{C}(\mathcal{G}^1),\phi,\lambda_{\textup{Ban}})$ is a $\mathrm{Aut}^{+}_{\textup{Ban}}(\clg)$-equivariant $\mathrm{C}^*$-correspondence over the $\mathrm{Aut}^{+}_{\textup{Ban}}(\clg)$-$\mathrm{C}^*$-algebra $(\mathrm{C}(\mathcal{G}^0),\alpha_{\textup{Ban}})$. Since the module dynamics corresponding to the gauge action is the scalar dynamics, i.e., $U_t=e^{it}$, $t \in \mathbb{R}$, by Theorem \ref{thm:kmsequiv}, the lifted $\mathrm{Aut}^{+}_{\textup{Ban}}(\clg)$-action $\omega$ of $\mathrm{Aut}^{+}_{\textup{Ban}}(\clg)$ on $\mathrm{C}^*(\clg)$ preserves a \textup{KMS} state $\varphi$ at the critical inverse temperature if and only if it preserves the tracial state $\tau$ on $\mathrm{C}(\mathcal{G}^0)$. This yields another proof of  \cite{Joardar2}*{Proposition 2.31}.
\end{remark} 

It can be proved that for the universal object $(\mathrm{Aut}^{+}_{\textup{Ban}}(\clg),\alpha_{\textup{Ban}},\lambda_{\textup{Ban}})$, we also have 
\[
(s_{\ast}\ot\mathrm{id}_{\mathrm{C}(G)})\alpha_{\textup{Ban}}(f)=\lambda_{Ban}(s_{\ast}(f)),\] for all $f \in \mathrm{C}(\mathcal{G}^0)$, which is the content of the next theorem.

\begin{proposition}
For the universal object $(\mathrm{Aut}^{+}_{\textup{Ban}}(\clg),\alpha_{\textup{Ban}},\lambda_{\textup{Ban}})$, we also have 
\[
(s_{\ast}\ot\mathrm{id}_{\mathrm{C}(G)})\alpha_{\textup{Ban}}(f)=\lambda_{Ban}(s_{\ast}(f)),
\] for all $f \in \mathrm{C}(\mathcal{G}^0)$
\end{proposition} 
	
\begin{proof} 
We begin by fixing a vertex $v \in \mathcal{G}^0$. Then if $s^{-1}(v)$ is nonempty, then the arguments used in the proof of the fact $(r_{\ast}\ot \mathrm{id}_{\mathrm{C}(G)})\alpha_{\textup{Ban}}(f)=\lambda_{\textup{Ban}}(r_{\ast}(f))$ can be repeated with obvious modifications. So we are left with the case when $s^{-1}(v)$ is empty. Then $s_{\ast}(\delta_{v})=0$ so that $\lambda_{\textup{Ban}}(s_{\ast}(\delta_{v}))=0$. On the other hand, denoting the generators of the $\mathrm{C}^*$-algebra $\mathrm{C}(\mathrm{Aut}^{+}_{\textup{Ban}}(\clg))$ by $\{q_{vw}\}_{v,w\in \mathcal{G}^0}$, as in Theorem \ref{thm:genrel}, we obtain, 
\[
(s_{\ast}\ot\mathrm{id}_{\mathrm{C}(G)})\alpha_{\textup{Ban}}(\delta_{v})=\sum_{w \in \mathcal{G}^0}s_{\ast}(\delta_{w})\ot q_{wv}=\sum_{f \in \mathcal{G}^1}(\sum_{s(f)=w}\delta_{f}\ot q_{s(f)v}).
\]But, for a fixed edge $f \in \mathcal{G}^1$, $q_{s(f)v}=\sum_{w \in \mathcal{G}^0}q_{s(f)v}q_{r(f)w}$. Since $s^{-1}(v)$ is empty, $q_{s(f)v}q_{r(f)w}=0$, which implies that $(s_{\ast}\ot \mathrm{id}_{\mathrm{C}(G)})\alpha_{\textup{Ban}}(\delta_{v})=0$ as well.
\end{proof}

The following lemma will be useful in proving Theorem \ref{thm:Biccorrespondence}.

\begin{lemma}
\label{lem:Ban_Bic}
Let $(G,\alpha,\lambda)$ be an object of the category $\clc_{\textup{Bic}}(\clg)$ \textup{(}Definition \textup{\ref{def:Bichcat})}. Then
\[
(r_\ast\otimes \mathrm{id}_{\mathrm{C}(G)})\alpha(f)=\lambda(r_{\ast}(f)),
\]as well as
\[
(s_\ast\otimes \mathrm{id}_{\mathrm{C}(G)})\alpha(f)=\lambda(s_{\ast}(f)),
\]for all $f \in \mathrm{C}(\mathcal{G}^0)$. 

Conversely, assume that one of $r$ and $s$ is injective. Let $(G,\alpha,\lambda)$ be a triple, where $G$ is a compact quantum group, $\alpha : \mathrm{C}(\mathcal{G}^0) \rightarrow \mathrm{C}(\mathcal{G}^0) \otimes \mathrm{C}(G)$ is a unital $*$-homomorphism and $\lambda : \mathrm{C}(\mathcal{G}^1) \rightarrow \mathrm{C}(\mathcal{G}^1) \otimes \mathrm{C}(G)$ is a linear map satisfying the following conditions.
\begin{itemize}
	\item The pair $(\mathrm{C}(\mathcal{G}^0),\alpha)$ is a $G$-$\mathrm{C}^*$-algebra;
	\item the triple $(\mathrm{C}(\mathcal{G}^1),\phi,\lambda)$ is a $G$-equivariant $\mathrm{C}^*$-correspondence over the $G$-$\mathrm{C}^*$-algebra $(\mathrm{C}(\mathcal{G}^0),\alpha)$;
	\item moreover,\[
		(r_\ast\otimes \mathrm{id}_{\mathrm{C}(G)})\alpha(f)=\lambda(r_{\ast}(f)),
		\]as well as
		\[
		(s_\ast\otimes \mathrm{id}_{\mathrm{C}(G)})\alpha(f)=\lambda(s_{\ast}(f)),
		\]for all $f \in \mathrm{C}(\mathcal{G}^0)$. 
\end{itemize}
Then the pair $(\mathrm{C}(\mathcal{G}^1),\lambda)$ is a $G$-$\mathrm{C}^*$-algebra and consequently, $(G,\alpha,\lambda)$ is an object of the category $\clc_{\textup{Bic}}(\clg)$.
\end{lemma}

\begin{proof}
We first assume that $(G,\alpha,\lambda)$ is an object in the category $\clc_{\textup{Bic}}(\clg)$. Let $\xi \in \mathrm{C}(\mathcal{G}^1)$ and $f \in \mathrm{C}(\mathcal{G}^0)$. It is clear that \[\xi f=\xi r_{\ast}(f),\] where $\xi f$ denotes the right module action of $f$ on $\xi$ and $\xi r_{\ast}(f)$ denotes the algebra multiplication in $\mathrm{C}(\clg^{1})$. Since $\lambda$ is a $G$-action, we obtain
\[
\begin{aligned}
\lambda(\xi f)={}&\lambda(\xi r_{\ast}(f))\\
={}&\lambda(\xi)\lambda(r_{\ast}(f)).
\end{aligned}
\]Similarly, decoding the right module action of $\mathrm{C}(\mathcal{G}^0)\ot \mathrm{C}(G)$ on $\mathrm{C}(\mathcal{G}^1)\ot \mathrm{C}(G)$, we get the following.
\[
\begin{aligned}
\lambda(\xi f)={}&\lambda(\xi)\alpha(f)\\
={}&(\lambda(\xi))(r_{\ast}\otimes \mathrm{id}_{\mathrm{C}(G)})\alpha(f).
\end{aligned}
\]In particular, choosing $\xi=1$, we obtain 
\[
\lambda(r_{\ast}(f))=(r_{\ast}\otimes \mathrm{id}_{\mathrm{C}(G)})\alpha(f),
\]for all $f \in \mathrm{C}(\mathcal{G}^0)$. Similarly, considering the left module action produces \[(s_{\ast}\otimes \mathrm{id}_{\mathrm{C}(G)})\alpha(f)=\lambda(s_{\ast}(f)).\]

For the converse, we assume that $r$ is injective. Then we note that it is enough to show that $\lambda$ is a $*$-homomorphism. As $r$ is injective, $r_{\ast} : \mathrm{C}(\mathcal{G}^0) \raro \mathrm{C}(\mathcal{G}^1)$ is a surjective $*$-homomorphism. So for $\xi,\eta \in \mathrm{C}(\mathcal{G}^1)$, 
\[
\lambda(\xi\eta)=\lambda(r_{\ast}(f) r_{\ast}(g)),
\]for some $f,g \in \mathrm{C}(\mathcal{G}^0)$. Since $r_{\ast}$ and therefore $(r_{\ast}\otimes \mathrm{id}_{\mathrm{C}(G)})$ are $*$-homomorphisms, we have
\[
\lambda(\xi\eta)=\lambda(r_{\ast}(fg))=(r_{\ast}\otimes \mathrm{id}_{\mathrm{C}(G)})(\alpha(f)\alpha(g)),
\]which is nothing but $(r_{\ast}\otimes \mathrm{id}_{\mathrm{C}(G)})\alpha(f)(r_{\ast}\otimes \mathrm{id}_{\mathrm{C}(G)})\alpha(g)$. Therefore
\[
\begin{aligned}
\lambda(\xi\eta)={}&(r_{\ast}\otimes \mathrm{id}_{\mathrm{C}(G)})\alpha(f)(r_{\ast}\otimes \mathrm{id}_{\mathrm{C}(G)})\alpha(g)\\
={}&\lambda(r_{\ast}(f))\lambda(r_{\ast}(g))\\
={}&\lambda(\xi)\lambda(\eta),
\end{aligned}
\]proving that $\lambda$ is a $*$-homomorphism. Similarly, if $s$ is injective, using the identity $(s_{\ast}\otimes \mathrm{id}_{\mathrm{C}(G)})\circ\alpha=\lambda \circ s_{\ast}$, one can prove that $\lambda$ is a $*$-homomorphism.
\end{proof}

Now we can turn to the proof of Theorem \ref{thm:Biccorrespondence}.

\begin{proof}[Proof of Theorem \ref{thm:Biccorrespondence}]
First, let us prove that $\mathrm{Aut}^{+}_{\textup{Bic}}(\clg)$ can be made into an object of the category $\clc_{\textup{Bic}}(\clg)$. We denote the generators of $\mathrm{Aut}^{+}_{\textup{Bic}}(\clg)$ by $(q_{vw})_{v,w\in \mathcal{G}^0}$ as in Remark \ref{rem:bichon}. let us define the maps $\lambda_{\textup{Bic}}$ and $\alpha_{\textup{Bic}}$ by equations \eqref{eq:lambda} and \eqref{eq:alpha}, respectively, i.e.,
\begin{align}
\lambda_{\textup{Bic}}(\delta_{e})=\sum_{f\in \mathcal{G}^1}\delta_{f}\ot q_{s(f)s(e)}q_{r(f)r(e)}, \quad e \in \mathcal{G}^1,
\end{align}
\begin{align}
\alpha_{\textup{Bic}}(\delta_{v})=\sum_{w\in \mathcal{G}^0}\delta_{w}\ot q_{wv}, \quad v \in \mathcal{G}^0.
\end{align}
Then $\lambda_{\textup{Bic}}$ is a $\mathrm{Aut}^{+}_{\textup{Bic}}(\clg)$-action on $\mathrm{C}(\mathcal{G}^1)$, by the definition of $\mathrm{Aut}^{+}_{\textup{Bic}}(\clg)$. Now $\mathrm{Aut}^{+}_{\textup{Bic}}(\clg)$ is a quantum subgroup of $\mathrm{Aut}^{+}_{\textup{Ban}}(\clg)$. We have already proved that $(\mathrm{C}(\mathcal{G}^1),\phi,\lambda_{\textup{Ban}})$ is a $\mathrm{Aut}^{+}_{\textup{Ban}}(\clg)$-equivariant $\mathrm{C}^*$-correspondence over the $\mathrm{Aut}^{+}_{\textup{Ban}}(\clg)$-$\mathrm{C}^*$-algebra $(\mathrm{C}(\mathcal{G}^0),\alpha_{\textup{Ban}})$. Therefore $(\mathrm{C}(\mathcal{G}^1),\phi,\lambda_{\textup{Bic}})$ is a $\mathrm{Aut}^{+}_{\textup{Bic}}(\clg)$-equivariant $\mathrm{C}^*$-correspondence over the $\mathrm{Aut}^{+}_{\textup{Bic}}(\clg)$-$\mathrm{C}^*$-algebra $(\mathrm{C}(\mathcal{G}^0),\alpha_{\textup{Bic}})$. This proves that $(\mathrm{Aut}^{+}_{\textup{Bic}}(\clg),\alpha_{\textup{Bic}},\lambda_{\textup{Bic}})$ is an object in the category $\clc_{\textup{Bic}}(\clg)$. 

To show that $(\mathrm{Aut}^{+}_{\textup{Bic}}(\clg),\alpha_{\textup{Bic}},\lambda_{\textup{Bic}})$ is a universal object, we pick an object $(G,\alpha,\lambda)$ from the category $\clc_{\textup{Bic}}(\clg)$. Then, by Lemma \ref{lem:Ban_Bic}, \[(r_{\ast}\otimes \mathrm{id}_{\mathrm{C}(G)})\circ\alpha=\lambda\circ r_{\ast}\]and consequently, $G$ becomes a quantum subgroup of $\mathrm{Aut}^{+}_{\textup{Ban}}(\clg)$. Since $\lambda$ is an action, the relations (\ref{eq:r4}) also hold in $\mathrm{C}(G)$. Hence, by the universality of $\mathrm{Aut}^{+}_{\textup{Bic}}(\clg)$, $G$ is a quantum subgroup of $\mathrm{Aut}^{+}_{\textup{Bic}}(\clg)$ as well. This shows that $(\mathrm{Aut}^{+}_{\textup{Bic}}(\clg),\alpha_{\textup{Bic}},\lambda_{\textup{Bic}})$ indeed is a universal object in the category $\clc_{\textup{Bic}}(\clg)$.
\end{proof}

The converse direction of Lemma \ref{lem:Ban_Bic} has the following corollary.

\begin{corollary}
\label{cor:BanBicsame}
Let $\clg$ be a finite, simple graph without any source. If either $r$ or $s$ is injective, then $\mathrm{Aut}^{+}_{\textup{Bic}}(\clg)$ is isomorphic to $\mathrm{Aut}^{+}_{\textup{Ban}}(\clg)$.
\end{corollary} 

\begin{proof}
We know that $\mathrm{Aut}^{+}_{\textup{Bic}}(\clg)$ is a quantum subgroup of $\mathrm{Aut}^{+}_{\textup{Ban}}(\clg)$. By Lemma \ref{lem:Ban_Bic}, we conclude that the pair $(\mathrm{C}(\mathcal{G}^1),\lambda)$ is a $\mathrm{Aut}^{+}_{\textup{Ban}}(\clg)$-$\mathrm{C}^*$-algebra and consequently, $(\mathrm{Aut}^{+}_{\textup{Ban}}(\clg),\alpha_{\textup{Ban}},\lambda_{\textup{Ban}})$ is an object of the category $\clc_{\textup{Bic}}(\clg)$. But then $\mathrm{Aut}^{+}_{\textup{Ban}}(\clg)$ is a quantum subgroup of $\mathrm{Aut}^{+}_{\textup{Bic}}(\clg)$. Therefore, $\mathrm{Aut}^{+}_{\textup{Ban}}(\clg)$ is isomorphic to $\mathrm{Aut}^{+}_{\textup{Bic}}(\clg)$.
\end{proof}

As an application, we shall give an example of a finite, simple graph $\clg$
without any source such that $\mathrm{Aut}^{+}_{\textup{Ban}}(\clg)$ is
genuinely quantum and is isomorphic to
$\mathrm{Aut}^{+}_{\textup{Bic}}(\clg)$. To that end, let us consider a graph
where the map $r$ is a bijection between the edge set and the vertex set. Note
that, if we assume the graph to be connected and without multiple edges or
loops, then we are not left with many choices. The graph essentially reduces to
an oriented $m$-gon whose quantum symmetry coincides with the classical symmetry
group which is $\mathbb{Z}/m\mathbb{Z}$ (\cite{Ban}*{Theorem 4.1}). Now, if we
consider disjoint union of $n$-number of such graphs, then the map $r$ is still
a bijection and we can apply Corollary \ref{cor:BanBicsame}. Combining this with
\cite{Bichonwreath}*{Theorem 4.2}, we obtain another proof of the following well-known result. But before stating it, let us remark that for a compact quantum group $G$, we denote the free wreath product of $G$ with the quantum permutation group of $n$-points $\textup{S}^+_n$ by $G\wr_{\ast} \textup{S}^+_n$; we refer the reader to \cite{Bichonwreath} for more details on this.

\begin{theorem}\cite{BV2009}*{Theorem 3.4}
Let $\clg$ be a graph which is a disjoint union of $n$-number of copies of oriented $m$-gons. Then $\mathrm{Aut}^{+}_{\textup{Ban}}(\clg)$ and $\mathrm{Aut}^{+}_{\textup{Bic}}(\clg)$ coincide and are isomorphic to
\[
\mathbb{Z}/m\mathbb{Z}\wr_{\ast}\textup{S}^+_n.
\]
\end{theorem} 


\section{Applications to quantum symmetries of graphs III: The case of multigraphs}
\label{sec:applicationsthree}
In this section, we allow the graphs to admit multiple edges. Let then
$\clg=(\mathcal{G}^0,\mathcal{G}^1,r,s)$ be a finite, directed graph, possibly
with multiple edges, but without any sources. We are interested in the category
$\clc_{\textup{Bic}}(\clg)$ from Definition \ref{def:Bichcat}. Although for a
simple graph, the above category has a smaller universal object compared to the
category of Definition \ref{def:Bancat}, for graphs with multiple edges, the
category $\clc_{\textup{Bic}}(\clg)$ is more relevant. For example, in the case
of the graph of Cuntz algebra, the $\textup{C}^*$\nobreakdash-algebra
$\textup{C}(\mathcal{G}^0)$ is $\mathbb{C}$ and hence has a trivial quantum
symmetry and consequently, the category $\clc_{\textup{Ban}}(\clg)$ only
consists of trivial actions of compact quantum groups. 

\begin{definition}
	\label{def:Bichcatmult} 
	Let $\mathcal{G}=(\mathcal{G}^1,\mathcal{G}^0,r,s)$ be a finite, directed graph, possibly with loops or multiple edges. We define the category $\clc_{\textup{Bic}}^{\textup{mult}}(\clg)$ as follows. 
	\begin{itemize}
		\item An object of $\clc_{\textup{Bic}}^{\textup{mult}}(\clg)$ is a triple $(G,\alpha,\lambda)$, where $G$ is a compact quantum group, $\alpha : \mathrm{C}(\mathcal{G}^0) \rightarrow \mathrm{C}(\mathcal{G}^0) \otimes \mathrm{C}(G)$ is a unital $*$-homomorphism and $\lambda : \mathrm{C}(\mathcal{G}^1) \rightarrow \mathrm{C}(\mathcal{G}^1) \otimes \mathrm{C}(G)$ is a linear map satisfying the following conditions.
		\begin{itemize}
			\item The pair $(\mathrm{C}(\mathcal{G}^0),\alpha)$ is a $G$-$\mathrm{C}^*$-algebra.
			\item the triple $(\mathrm{C}(\mathcal{G}^1),\phi,\lambda)$ is a $G$-equivariant $\mathrm{C}^*$-correspondence over the $G$-$\mathrm{C}^*$-algebra $(\mathrm{C}(\mathcal{G}^0),\alpha)$;
			\item the pair $(\mathrm{C}(\mathcal{G}^1),\lambda)$ is a $G$-$\mathrm{C}^*$-algebra and the action $\lambda$ is faithful.
		\end{itemize}
		\item Let $(G_1,\alpha_1,\lambda_1)$ and $(G_2,\alpha_2,\lambda_2)$ be two objects of the category $\clc_{\textup{Bic}}^{\textup{mult}}(\clg)$. A morphism $f : (G_1,\alpha_1,\lambda_1) \rightarrow (G_2,\alpha_2,\lambda_2)$ in $\clc_{\textup{Bic}}^{\textup{mult}}(\clg)$ is again by definition a Hopf $*$-homomorphism $f : \mathrm{C}(G_2) \rightarrow \mathrm{C}(G_1)$ such that 
		\begin{itemize}
			\item $(\mathrm{id}_{\mathrm{C}(\mathcal{G}^0)} \otimes f)\alpha_2=\alpha_1$;
			\item $(\mathrm{id}_{\mathrm{C}(\mathcal{G}^1)} \otimes f)\lambda_2=\lambda_1$.
		\end{itemize}
	\end{itemize}
\end{definition}
\begin{theorem}
\label{thm:Bic_multi}
For a finite, directed graph $\clg$ which has multiple edges but has no source, the category $\clc_{\textup{Bic}}^{\textup{mult}}(\clg)$ admits a universal object.
\end{theorem} 

We shall not give a proof of the above theorem as we are mainly interested in concrete examples of universal objects in the above category. In fact, we shall produce an example illustrating a general strategy to identify quantum subgroups of certain compact quantum groups. Also, we are aware that a notion of the quantum automorphism group of multigraphs is being formulated (\cite{HG2022}) and we believe that it will have some point of contact with the results obtained in this section. But we must wait for its final formulation to appear to say anything precise in this direction.
\begin{remark}
\label{rem:subobject_multiple edges}
We remark that for any object $(G,\alpha,\lambda)$ in the category $\clc_{\textup{Bic}}^{\textup{mult}}(\clg)$, $G$ is of Kac type. Also the $G$-action $\alpha$ preserves the restriction of the distinguished \textup{KMS} state $\varphi$ on $\mathrm{C}(\mathcal{G}^0)$. Therefore the universal object in the category $\clc_{\textup{Bic}}^{\textup{mult}}(\clg)$ is a subobject of the universal object in the category of compact quantum groups acting faithfully on $\mathrm{C}^*(\clg)$, preserving the distinguished \textup{KMS} state $\varphi$ at the critical inverse temperature, whenever $\varphi$ exists.
\end{remark}   

 Let $\clg$ be the graph with one vertex and $n$ number of loops at the vertex; $\clg^{\oplus m}$ be the disjoint union of $m$-copies of $\clg$. We note that the corresponding graph $\mathrm{C}^*$-algebra is $\oplus_{i=1}^{m}\clo_{n}$, where $\clo_{n}$ is the Cuntz algebra generated by $n$ number of partial isometries. So $\clg^{\oplus m}$ has $m$-vertices with each vertex emitting and receiving $n$ number of edges. We shall compute the universal object of the category $\clc_{\textup{Bic}}^{\textup{mult}}(\clg^{\oplus m})$.\\   

Let us introduce now some notations. We denote the vertex set of $\clg^{\oplus m}$ by $\mathcal{G}^0$; the edge set by $\mathcal{G}^1$, as usual. We denote the vertices by $\{v_{i} \mid i=1,\dots,m\}$ and edges by $\{e^{(i)}_{j} \mid i=1,\dots,m, j=1,\dots,n\}$. $e^{(i)}_{j}$ denotes the $j$-th loop in the $i$-th copy of $\clg$. We denote by $\delta_{i}$ the function on $\mathcal{G}^0$ that takes the value $1$ on $v_{i}$ and zero on the other vertices; we denote by $\xi_{i}^{(k)}$ the function on $\mathcal{G}^1$ that takes the value $1$ on $e^{(k)}_{i}$ and $0$ on the other edges. The result in the following theorem resembles the conclusion of \cite{Bichonwreath}*{Theorem 4.2}.

\begin{theorem}
\label{thm:qautwreath} 
Let us denote by $(G_{\textup{univ}},\alpha_{\textup{univ}},\lambda_{\textup{univ}})$ the universal object of the category $\clc_{\textup{Bic}}^{\textup{mult}}(\clg^{\oplus m})$. Then $G_{\textup{univ}}$ is isomorphic to $\textup{S}_{n}^{+}\wr_{\ast} \textup{S}^+_m$.
\end{theorem}

\begin{proof}
We begin by introducing some more notations. We shall denote the generators of $\textup{S}^{+}_{n}$ by $u_{ij}$, where $i,j=1,\dots,n$ and the generators of $\textup{S}^{+}_{m}$ by $v_{kl}$, where $k,l=1,\dots,m$. We also denote the embedding of the $l$-th copy of $\textup{S}^{+}_{n}$ in $\textup{S}^{+}_{n}\underbrace{\ast\dots\ast}_{m\textup{-times}}\textup{S}^{+}_{n}$ by $\nu_{l}$. With these notations, the compact quantum group $\textup{S}^{+}_{n}\wr_{\ast} \textup{S}^{+}_{m}$ can be made into an object of the category $\clc_{\textup{Bic}}^{\textup{mult}}(\clg^{\oplus m})$ as follows. We define
\begin{align}
\lambda_{\textup{Bic}}(\xi^{i}_{(k)})=\sum_{j,l}\xi_{(l)}^{j}\ot \nu_{l}(u_{ij})v_{lk},
\end{align}for $i=1,\dots, n$ and $k=1,\dots,m$; and
\begin{align}
\alpha_{\textup{Bic}}(\delta_{k})=\sum_{l=1}^{m}\delta_{l}\ot v_{lk}, 
\end{align}for $k=1,\dots,m$. We leave the proof of the fact that $(\textup{S}^{+}_{n}\wr_{\ast} \textup{S}^{+}_{m},\alpha_{\textup{Bic}},\lambda_{\textup{Bic}})$ is indeed an object of the category $\clc_{\textup{Bic}}^{\textup{mult}}(\clg^{\oplus m})$ to the reader; see also \cite{Bichonwreath}. We only prove that the triple $(\textup{S}^{+}_{n}\wr_{\ast} \textup{S}^{+}_{m},\alpha_{\textup{Bic}},\lambda_{\textup{Bic}})$ is the universal object in the category $\clc_{\textup{Bic}}^{\textup{mult}}(\clg^{\oplus m})$. 

To that end, let $(G,\alpha,\lambda)$ be an object of the category $\clc_{\textup{Bic}}^{\textup{mult}}(\clg^{\oplus m})$. Let, moreover, $\lambda$ and $\alpha$ be given by
\begin{align}
\lambda(\xi_{(k)}^{i})=\sum_{j,l}\xi_{(l)}^{j}\ot q_{l(j)k(i)}
\end{align}
\begin{align}
\alpha(\delta_{k})=\sum_{l}\delta_{l}\ot x_{lk},
\end{align}with the notations as in the paragraph just above the statement of the theorem. We note that since $(\mathrm{C}(\mathcal{G}^1),\lambda)$ is a $G$-$\mathrm{C}^*$-algebra, we have the following relations.
\begin{align}
\kappa(q_{l(j)k(i)})=q_{k(i)l(j)},
\end{align}
\begin{align}\label{eq:7.4}
\sum_{k,i}q_{l(j)k(i)}=1,
\end{align}where $\kappa$ denotes the antipode of $G$. Also, since, 
\[r_{\ast}(\delta_{k})=\sum_{i=1}^{n}\xi^{(k)}_{i}\]for all $k=1,\dots,m$ and \[(r_{\ast}\otimes \mathrm{id}_{\mathrm{C}(G)})\alpha(\delta_{k})=\lambda(r_{\ast}(\delta_{k})),\]we have, for $k,l=1,\dots,m$,
\[
x_{lk}=\sum_{i=1}^{n}q_{l(j)k(i)},
\]for all $j=1,\dots,n$. Clearly $x_{lk}$, for $k,l=1,\dots,m$ satisfy the relations in $\textup{S}_{m}^{+}$. So in particular, 
\[\sum_{l=1}^{m}x_{lk}=1.\]We now define \[u_{ji}^{(l)}=\sum_{k=1}^{m}q_{l(j)k(i)},\] for $l=1,\dots,m$. Hence by \eqref{eq:7.4}, for each $l=1,\dots,m$, 
\begin{align} 
\sum_{i=1}^{n}u^{(l)}_{ji}=1.
\end{align}Now, using \[\sum_{l}x_{lk}=1,\] we get \[\sum_{l,i}q_{l(j)k(i)}=1,\]which after applying $\kappa$ gives us
\[
\sum_{l,i} q_{k(i)l(j)}=1,
\]i.e., after interchanging $k$ with $l$ and $i$ with $j$,
\begin{align}
\sum_{j=1}^{n}u^{(l)}_{ji}=1,
\end{align}for all $i,l$. An easy calculation, using \[\xi^{k}_{i}\xi^{k^{\prime}}_{i^{\prime}}=\delta_{k,k^{\prime}}\delta_{i,i^{\prime}}\xi^{(k)}_{i},\] in the $\mathrm{C}^*$-algebra $\mathrm{C}(\mathcal{G}^1)$ reveals that for all $i,j,k,l,k^{\prime},i^{\prime}$,
\begin{align}
\label{eq:14} 
q_{l(j)k(i)}q_{l(j)k^{\prime}(i^{\prime})}=\delta_{k,k^{\prime}}\delta_{i,i^{\prime}}q_{l(j)k(i)}.
\end{align}Hence, for all $l=1,\dots,m$,
\begin{align}
u^{(l)}_{ji}u^{(l)}_{ji^{\prime}}=\delta_{i,i^{\prime}}u^{(l)}_{ji}.\end{align}Now for $k \neq l$, \[\xi^{(k)}_{i}\delta_{l}=0,\]for all $i=1,\dots,n$. Therefore
\begin{align*}
\lambda(\xi^{(k)}_{i}\delta_{l})=0,
\end{align*}which implies,
\begin{align*}
(\sum_{p,r}\xi^{(p)}_{r}\ot q_{p(r)k(i)})(\sum_{s}\delta_{s}\ot x_{sl})=0,
\end{align*}i.e.,
\begin{align*}
(\sum_{p,r}\xi^{(p)}_{r}\ot q_{p(r)k(i)}x_{pl})=0.
\end{align*}Hence we have \[q_{p(r)k(i)}(\sum_{i^{\prime}}q_{p(j)l(i^{\prime})})=0.\]For $j=r$, as $k \neq l$, \[q_{p(r)k(i)}q_{p(j)l(i^{\prime})}=0\] for all $i^{\prime}$. Therefore, for $j \neq r$,
\begin{align}
\label{eq:15} 
q_{p(r)k(i)}(\sum_{i^{\prime}}q_{p(j)l(i^{\prime})})=0.
\end{align}Using \[q_{p(j)l(i^{\prime})}q_{p(j)l(i^{\prime\prime})}=\delta_{i^{\prime},i^{\prime\prime}}q_{p(j)l(i^{\prime})}\] for all $p,j,l,i^{\prime}$, we get
\begin{align}
\label{eq:16}
q_{p(r)k(i)}q_{p(j)l(i)}=0,
\end{align}whenever $k \neq l$ and $j \neq r$. Now for any $l=1,\dots,m$ and any $i,r,s=1,\dots,n$,
\begin{align*}
u^{(l)}_{ri}u^{(l)}_{si}=\sum_{k,k^{\prime}}q_{l(r)k(i)}q_{l(s)k^{\prime}(i)},
\end{align*}which, for $r=s$ is equal to \[\sum_{k}q_{l(r)k(i)}.\]Hence, we have that \[(u^{(l)}_{ri})^{2}=u^{(l)}_{ri}.\] For $r \neq s$, using \eqref{eq:16}, we get
\begin{align*}
u^{(l)}_{ri}u^{(l)}_{si}=\sum_{k=1}^{m}q_{l(r)k(i)}q_{l(s)k(i)}.
\end{align*}From \eqref{eq:14}, for any $l,r,k,i,s$ with $r \neq s$,
\[
q_{k(i)l(r)}q_{k(i)l(s)}=0.
\]Applying $\kappa$ to the above identity, we deduce that \[q_{l(s)k(i)}q_{l(r)k(i)}=0\] for all $i,k,l$ whenever $r \neq s$. Therefore, 
\[
u^{(l)}_{ri}u^{(l)}_{si}=\delta_{r,s}u^{(l)}_{ri}, 
\]for all $l=1,\dots,m$ and for all $i,r,s=1,\dots,n$. So for each $l=1,\dots,m$, $u^{(l)}_{ji}$, $i,j=1,\dots,n$ satisfy
\begin{align}
\begin{aligned}
\sum_{i=1}^{n}u^{(l)}_{ij}={}&\sum_{j=1}^{n}u^{(l)}_{ij}=1;\\
\ u^{(l)}_{ij}u_{ik}^{(l)}={}&\delta_{j,k}u^{}_{ij};\\
u^{(l)}_{kj}u^{(l)}_{ij}={}&\delta_{k,i}u_{ij}^{(l)}. 
\end{aligned}
\end{align}Also $x_{kl}$ for $k,l=1,\dots,m$ satisfy the relations of $\textup{S}_{m}^{+}$. Then by the universal property of the free product, we get a Hopf $*$-homomorphism
\[f : \mathrm{C}(\textup{S}^+_n\underbrace{\ast\dots\ast}_{m\textup{-times}}\textup{S}_{n}^{+}\ast \textup{S}^+_m) \raro \mathrm{C}(G)\] such that \[f(\nu_{l}(u_{ij}))=u^{(l)}_{ij}\] and \[f(v_{kl})=x_{kl}\] where, we recall that $u_{ij}$ for $i,j=1,\dots,n$ are the generators of $\textup{S}^+_n$ and $v_{kl}$ for $k,l=1,\dots,m$ are the generators of $\textup{S}^+_m$. Moreover,
\begin{align}
\begin{aligned}
f(\nu_{l}(u_{ij})v_{lk})={}& u^{(l)}_{ij}x_{lk}\\
={}&(\sum_{k^{\prime},i^{\prime}}q_{l(i)k^{\prime}(j)}q_{l(i)k(i^{\prime})})\\
={}& q_{l(i)k(j)}.
\end{aligned}
\end{align}But $f(v_{lk}\nu_{l}(u_{ij}))$ is also $q_{l(i)k(j)}$. So by the definition of the free wreath product, $f$ descends to a surjective Hopf $*$-homomorphism \[f : \mathrm{C}(\textup{S}^+_n\wr_{\ast} \textup{S}^+_m) \rightarrow \mathrm{C}(G),\] intertwining the actions. Hence $(\textup{S}^+_n\wr_{\ast} \textup{S}^+_m,\alpha_{\textup{Bic}},\lambda_{\textup{Bic}})$ is the universal object in the category $\mathcal{C}_{\textup{Bic}}^{\textup{mult}}(\clg^{\oplus m})$, and this completes the proof.
\end{proof}

\begin{remarks}\hfill
\label{rem:qsubgroupremark}
\begin{itemize}
	\item For $m=1$, the graph $\clg^{\oplus 1}$ is the graph with one vertex and $n$ number of loops attached with it. The universal object in the category $\clc_{\textup{Bic}}^{\textup{mult}}(\clg^{\oplus 1})$ is $\textup{S}^+_n$ by Theorem \ref{thm:qautwreath}. Now the universal object in the category of compact quantum groups acting faithfully on $\clo_{n}$ and preserving the \textup{KMS} state, which is a distinguished one, is the free unitary quantum group $\textup{U}_{n}^{+}$. Hence by Remark \ref{rem:subobject_multiple edges}, $\textup{S}^+_n$ is realized as a quantum subgroup of $\textup{U}^{+}_{n}$.
	\item Similarly, for $n=1$, the graph $\clg^{\oplus m}$ is the disjoint union of $m$-loops. By \cite{Joardar1}*{Proposition 4.2}, the quantum automorphism group of the graph $\mathrm{C}^*$-algebra corresponding to the graph $\clg^{\oplus m}$ preserving the distinguished \textup{KMS} state is $\textup{H}_{m}^{\infty +}$. Hence again by Remark \ref{rem:subobject_multiple edges}, the universal object in the category  $\clc_{\textup{Bic}}^{\textup{mult}}(\clg^{\oplus m})$  realizes $\textup{S}^+_m$ as a quantum subgroup of $\textup{H}^{\infty +}_{m}$.
\end{itemize}  
\end{remarks}

\end{document}